\documentclass[12pt]{amsart}

\usepackage{amssymb}

\makeatletter
\@addtoreset{equation}{section}
\makeatother

\theoremstyle{plain}
\newtheorem{theorem}{Theorem}[section]
\newtheorem{lemma}[theorem]{Lemma}
\newtheorem{proposition}[theorem]{Proposition}
\newtheorem{corollary}[theorem]{Corollary}
\newtheorem{question}[theorem]{Question}
      
\theoremstyle{definition}
\newtheorem{definition}[theorem]{Definition}
\newtheorem{remark}[theorem]{Remark}
\newtheorem{example}[theorem]{Example}

\newcommand{\Cee}{{\mathbb C}}
\newcommand{\Ree}{{\mathbb R}}
\newcommand{\Tee}{{\mathbb T}}
\newcommand{\En}{{\mathbb N}}
\newcommand{\Zee}{{\mathbb Z}}

\newcommand{\fA}{{\mathcal A}}
\newcommand{\fB}{{\mathcal B}}
\newcommand{\fC}{{\mathcal C}}
\newcommand{\fE}{{\mathcal E}}
\newcommand{\fF}{{\mathcal F}}
\newcommand{\fH}{{\mathcal H}}
\newcommand{\fI}{{\mathcal I}}
\newcommand{\fJ}{{\mathcal J}}
\newcommand{\fK}{{\mathcal K}}
\newcommand{\fL}{{\mathcal L}}
\newcommand{\fN}{{\mathcal N}}
\newcommand{\fP}{{\mathcal P}}
\newcommand{\fS}{{\mathcal S}}
\newcommand{\fT}{{\mathcal T}}

\newcommand{\fW}{{\mathcal W}}
\newcommand{\fX}{{\mathcal X}}

\newcommand{\alp}{\alpha}
\newcommand{\del}{\delta}
\newcommand{\eps}{\varepsilon}
\newcommand{\sig}{\sigma}
\newcommand{\vpi}{\varpi}

\newcommand{\norm}[1]{\left\|#1\right\|}
\newcommand{\wbar}[1]{\overline{#1}} 
\newcommand{\til}[1]{\tilde{#1}} 
 
\newcommand{\what}[1]{\widehat{#1}}

\newcommand{\id}{\mathrm{id}}
\newcommand{\idem}{\mathrm{E}}
\newcommand{\ze}{\mathrm{ZE}}

\newcommand{\edlg}{\mathrm{EdLG}} 
\newcommand{\Aff}{\mathrm{Aff}}
\newcommand{\spn}{\mathrm{span}}
\newcommand{\un}{\mathrm{U}}
 
\newcommand{\bl}{\mathrm{L}}
 \newcommand{\fal}{\mathrm{A}}
 \newcommand{\fsal}{\mathrm{B}}
 \newcommand{\ideal}{\mathrm{I}}
 \newcommand{\pd}{\mathrm{P}}
 \newcommand{\wstar}{\mathrm{W}^*}
 \newcommand{\zp}{\mathrm{ZProj}}

\begin{document}
 
 
\author{Nico Spronk}
\address{Department of Pure Mathematics, University of Waterloo, Waterloo, Ontario, N2L3G1,Canada}
\email{nspronk@uwaterloo.ca}

\title[Topologies, idempotents and ideals]{Weakly almost periodic
topologies, idempotents and ideals}

\begin{abstract}
Let $(G,\tau_G)$ be a topological group.  We establish relationships between
weakly almost periodic topologies on $G$ coarser than $\tau_G$, central idempotents
in the weakly almost periodic compactification $G^\fW$, and certain ideals in the 
algebra of weakly almost periodic functions $\fW(G)$.  We gain decompositions
of weakly almost periodic representations, generalizing many from the literature.
We look at the role of pre-locally compact topologies, unitarizable topologies, and
extend or decompositions to Fourier-Stieltjes algebras $\fsal(G)$.
\end{abstract}


\thanks{This research was partially supported by an NSERC Discovery Grant.}


 \date{\today}
 
 \maketitle
 

\section{Introduction}

\subsection{Some classical decompositions}
Let $G$ be a topological group, $\fW(G)$ denote the C*-algebra of weakly almost periodic
functions, and $\fA\fP(G)$ the algebra of almost periodic functions.
The following is due to Eberlein \cite{eberlein} for locally compact abelian groups 
de Leeuw and Glicksberg \cite[5.11]{deleeuwg}, generally:
\begin{equation}\label{eq:meandecomp}
\fW(G)=\fA\fP(G)\oplus\fW_0(G)
\end{equation}
where $\fW_0(G)=\{u\in\fW(G):m(|u|)=0\}$ for the unique invariant mean $m$ on $\fW(G)$.
On a similar note, if $\pi:G\to\fB(\fX)$ is a weak operator continuous representation on a Banach space
which is weakly almost periodic (i.e.\ the weak operator closure, $\wbar{\pi(G)}^{wo}$, is weak operator compact)
then there is a decomposition
\begin{equation}\label{eq:repdecomp}
\fX=\fX^\pi_{ret}\oplus\fX^\pi_{wm}
\end{equation}
where
\begin{align*}
\fX^\pi_{ret}&=\left\{\xi\in\fX:\text{ for each }\eta\in\wbar{\pi(G)\xi}^w\text{ we have }\xi\in\wbar{\pi(G)\eta}^w\right\} \\
\fX^\pi_{wm}&=\left\{\xi\in\fX: 0\in\wbar{\pi(G)\xi}^w\right\}.
\end{align*}
We note that $\wbar{\pi(G)\xi}^w=\wbar{\pi(G)}^{wo}\xi$.  The elements of $\fX^\pi_{ret}$ are called {\it return
vectors}, while elements of $\fX^\pi_{wm}$ are called {\it weakly mixing vectors}.
See results of Bergelson and Rosenblatt \cite{bergelsonr}, Dye \cite{dye}, or 
Jacobs \cite{jacobs}; in the latter, elements of $\fX^\pi_{wm}$ are called ``flight vectors".

\subsection{Plan}
The goal of this paper is to devise a generalisation of (\ref{eq:meandecomp}) and 
(\ref{eq:repdecomp}).  The essential idea, is that we view these decompositions
as in terms of elements continuous with respect to the almost periodic, or Bohr, topology,
and view the complementary ideal/subspace as consisting of a space of elements which
are highly discontinuous with respect to this topology.  Hence we wish to find more general topologies
with respect to which we may obtain such decompositions.

Our first major theorem is Theorem \ref{theo:galois} which establishes the relationship
between weakly almost periodic topologies on $G$ and central idempotents 
in the weakly almost periodic compactification.  This theorem establishes a 
Galois connection between these two sets, and allows us to distinguish
a complete sublattice of ``maximally cocompact topologies", whose minimal elements
is the almost periodic topology.  These are used in the next major result, 
Theorem \ref{theo:main} to determine some special ideals of $\fW(G)$,
generalizing $\fW_0(G)$, which we call Eberlein-de Leeuw-Glicksberg (E-dL-G) ideals.
This discontinuous nature of elements of these E-dL-G ideals is indicated
in Lemma \ref{lem:bergelsonr}.  Finally, the decompositions of weakly almost
periodic representations in given in Theroem \ref{theo:contdecomp}.
In ensuing corollaries, we show how this naturally lends itself to various 
generalization of decompositions
in the literatures, including ones of de Leeuw and Glicksberg, Segal and von Neumann, and
of Lau and Losert.

In later sections we discuss pre-locally compact topologies and unitarizable topologies.
We answer a question left open in the investigation \cite{ilies,ilies1}
about the semi-lattice structure of maximally cocompact
pre-locally compact topologies, and give a topological
view of the spine algebra.  We give a version of our first main theorem, relating
unitarizable topologies to central idempotents in the universal unitary compactification
of $G$ (called the ``Eberlein compactification" in, for example, \cite{spronks}).
We redevelop the basic theory of the Fourier-Stieltjes algebra, for a unitarizable group.
We show that this algebra also admits E-dL-G ideals.  We close by indicating
a cross-section of examples from the literature, illustrating relationships between our objects.

\section{Weakly almost periodic topologies and central idempotents}

\subsection{Weakly almost periodic functions and semitopological semigroups}
Let $G=(G,\tau_G)$ be a topological group.  Hence we think of $\tau_G$
as the ``ambient topology".   We let $\fC_\beta(G)$ denote the C*-algebra
of continuous bounded complex functions on $G$ with uniform norm $\norm{\cdot}_\infty$.
For $s$ in $G$ and $f$ in $\fC_\beta(G)$ we let $u\cdot s(t)=u(st)$  for $t$ in $G$.  
We define the {\it almost periodic} and {\it weakly almost peroidc} functions on $G$ by
\begin{align*}
\fA\fP(G)&=\left\{u\in\fC_\beta(G):\wbar{u\cdot G}\text{ is compact in }\fC_\beta(G)\right\}\text{, and} \\
\fW(G)&=\left\{u\in\fC_\beta(G):\wbar{u\cdot G}^w\text{ is weakly compact in }\fC_\beta(G)\right\}.
\end{align*}
Each set is known to be a closed C*-subalgebra of $\fC_\beta(G)$, and may be 
equivalently defined as those $u$ for which $\wbar{G\cdot u}$ is closed, respectively,
$\wbar{G\cdot u}^w$ is weakly closed, where
$s\cdot u(t)=u(ts)$, for all $s$ and $t$ in $G$.  See the standard reference texts
\cite{burckel,ruppertB,berglundjm} for other facts in this paragraph.
We let for $\fA=\fA\fP$ or $\fW$,
$\eps^\fA:G\to\fA(G)^*$ be the evaluation map.  Then
\[
G^{\fA\fP}=\wbar{\eps^{\fA\fP}(G)}^{w^*},\;
G^\fW=\wbar{\eps^\fW(G)}^{w*}
\]
are the Gelfand spectrums of $\fA\fP(G)$, respectively, $\fW(G)$. Under the weak*-topology, 
$G^{\fA\fP}$ is a compact topological group, and $G^\fW$ is a compact semi-topological semigroup ---
i.e.\ maps $s\mapsto st,s\mapsto ts$ are continuous for each $t$ ---
each in a manner uniquely extending the group structure on $\eps^\fA(G)$.
We note the {\it universal property} of $(\eps^\fW,G^\fW)$.  If $(\del,S)$ is a {\it semi-topological
compactification} of $G$, i.e.\ $S$ is a semi-topological semigroup and $\del:G\to S$ a continuous
homomorphism with dense range, then there is a unique continuous map
\[
\del^\fW:G^\fW\to S\text{ with }\del^\fW\circ\eps^\fW=\del\text{ on }G
\]
which must hence be a surjective homomorphism.
This map is the adjoint of the inclusion $\fC(S)\circ\del\hookrightarrow\fW(G)$, restricted
to characters.
In particular, if $\eta:G\to H$ is a continuous homomorphism between topological groups
with dense range, then it induces a continuous surjective homomorphism
\begin{equation}\label{eq:tileta}
\eta^\fW:G^\fW\to  H^\fW
\end{equation}
which satisfies $\eta^\fW\circ\eps^\fW=\eps^\fW\circ\eta$ on $G$.

Let $G$ be a group and $\tau$ be a group topology on $G$.  A filter $\fF$ on $G$
is {\it $\tau$-Cauchy} if for any $\tau$-neighbourhood $U$ of $e$, there is $F_U$ in $\fF$
for which $s^{-1}t,st^{-1}\in U$ for $s,t$ in $F_U$.  Then we say that $G$ is {\it complete}
with respect to (the two-sided uniformity generated by) $\tau$, if any $\tau$-Cauchy filter converges.

Let $S$ denote any semigroup.  We denote the set of idempotents by $\idem(S)$.
If $e\in\idem(S)$ we let the {\it intrinsic group} at $e$ be given by
\[
S(e)=\{s\in eSe:e\in (seSe)\cap(eSes)\}.
\]

\begin{theorem}\label{theo:ruppert}
{\rm \cite[II.4.4,II.4.6]{ruppertB}}
If $S$ is  a semi-topological semigroup  
and $e\in\idem(S)$, then the group $S(e)$, with relativized topology,
is a complete topological group.
\end{theorem}

That $S(e)$ is a topological group is a significant consequence of the celebrated 
joint continuity theorem due separately to  Ellis and Lawson.  

\subsection{Weakly almost periodic topologies}
Given a set $X$ and any family of functions $\fF=\{f_i:X\to (Y_i,\tau_i)\}_{i\in I}$, where
each $(Y_i,\tau_i)$ is a topological space, we let $\sig(X,\fF)$ denote the {\it initial topology}
on $X$ generated by $\fF$.  

Let $G=(G,\tau_G)$ be a topological group.  

\begin{definition}
We denote the family of {\it weakly almost periodic
topologies} on $G$ by
\[
\fT(G)=\{\tau\subseteq\tau_G:\tau\text{ is a group topology, with }\tau=\sig(G,\fW^\tau(G))\}
\]
where $\fW^\tau(G)=\{u\in\fW(G):u\text{ is }\tau\text{-continuous}\}$. 
\end{definition}

Let $G_d$ denote $G$ qua discrete group.  Then it follows form the Hahn-Banach theorem that
$\fT(G)=\{\tau\in\fT(G_d):\tau\subseteq\tau_G\}$.  We make no assumption
that each $\tau$ in $\fT(G)$ is Hausdorff, the value of which should be apparent from
the next example.  

\begin{example} \label{ex:ap}  The
{\it almost periodic}, $\tau_{ap}=\sig(G,\fA\fP(G))$, is frequently not Hausdorff.
\end{example}

We shall henceforth assume that the following holds, unless indicated otherwise.

\begin{definition}  We say that $G=(G,\tau_G)$ is {\it weakly almost periodic} provided that
$\tau_G$ is Hausdorff and $\tau_G\in\fT(G)$, i.e.\ $\tau_G=\sig(G,\fW(G))$.
\end{definition}

Hence $G$ being weakly almost periodic entails that the map $\eps^\fW:G\to G^\fW$ is a 
homeomorphism onto its range.  A locally compact group
is always weakly almost periodic, since $\fC_0(G)\subseteq\fW(G)$.
We shall give examples of weakly almost periodic, and non-weakly almost periodic groups
in \ref{sec:comparison}.

Given $\tau$ in $\fT(G)$, we let $\eps^{\fW^\tau}:G\to\fW^\tau(G)^*$ be the evaluation
mapping, and $G^{\fW^\tau}=\wbar{\eps^{\fW^\tau}(G)}^{w*}$,  which is
naturally isomorphic to the weakly almost periodic
compactification of $(G,\tau)$.  Let 
\[
G_\tau=G^{\fW^\tau}(\eps^{\fW^\tau}(e_G))\text{ and }\eta_\tau:G\to G_\tau,\;
\eta_\tau(s)=\eps^{\fW^\tau}(s).
\]
In light of Theorem \ref{theo:ruppert}, 
the pair $(\eta_\tau,G_\tau)$ will be referred to as the {\it $\tau$-completion} of $G$.
More precisely, the group $G_\tau$ is the (unique, up to topological isomorphism) completion of
$G/\wbar{\{e_G\}}^\tau$ with respect to the two-sided uniformity from the induced topology.
Our assumption that $G$ is weakly almost periodic entails that
$G_{\tau_G}=G^\fW(\eps^\fW(e_G))$ is the completion of $G$.

By construction, each completion $G_\tau$ is a weakly almost periodic group.
In particular we have
\begin{equation}\label{eq:WGtau}
\fW(G_\tau)\circ \eta_\tau=\fW^\tau(G).
\end{equation}
Also, we let $\eps^\fW_\tau:G_\tau\to\fW(G_\tau)^*$ denote the evaluation map.

Notice that $\ker\eta_\tau=\wbar{\{e_G\}}^\tau$.
It follows that if $\tau_1,\tau_2\in\fT(G)$, then $\tau_1\subseteq\tau_2$
if and only if $\ker\eta_{\tau_1}\supseteq\ker\eta_{\tau_2}$ and $\ker\eta_{\tau_1}$
is $\tau_2$-closed.  Hence there is a unique continuous homomorphism
\[
\eta^{\tau_2}_{\tau_1}:G_{\tau_2}\to G_{\tau_1}
\text{ for which }\eta^{\tau_2}_{\tau_1}\circ\eta_{\tau_2}=\eta_{\tau_1}.
\]

\begin{definition}
If $\tau_1,\tau_2\in\fT(G)$, we say that $\tau_1$ is {\it co-compactly contained}
in $\tau_2$, written $\tau_1\subseteq_c\tau_2$, provided 

$\bullet$ $\tau_1\subseteq\tau_2$,

$\bullet$ $\ker\eta^{\tau_2}_{\tau_1}$ is compact, and 

$\bullet$ $\eta^{\tau_2}_{\tau_1}:G_{\tau_2}\to G_{\tau_1}$ is open.
\end{definition}

In particular, if $\tau_1\subseteq_c\tau_2$, then $G_{\tau_1}
\cong G_{\tau_2}/\ker\eta^{\tau_2}_{\tau_1}$, homeomorphically.
In the terminology of \cite{ilies1}, if $\tau_1\subseteq_c\tau_2$,
say that $\tau_1$ is a {\it quotient topology} of $\tau_1$.
In an amusing play on words, the relation of co-compact containment is also called ``co-Cauchy
containment", in \cite{ruppert}.  In fact, the relation there is slightly more general than is presented 
here.  This terminology is motivated by the following, which (for abelian groups) is proved in the same article.
We include our own proof for completeness of presentation.

\begin{lemma}\label{lem:cocompact}
Let $\tau_1\subseteq\tau_2$ in $\fT(G)$.  Then $\tau_1\subseteq_c\tau_2$ if and only if
every $\tau_1$-Cauchy filter on $G$ admits a $\tau_2$-Cauchy refinement.
\end{lemma}

\begin{proof}
Given a filter base $\fE$ on a set $X$, we let $\langle\fE\rangle$ denote the filter on $X$ generated 
by $\fE$.  If $\eta:X\to Y$, $\eta(\fE)=\{\eta(E):E\in\fE\}$; and if $\eta:Y\to X$ then
$\eta^{-1}(\fE)=\{\eta^{-1}(E):E\in\fE\}$.  For $j=1,2$, a filter $\fF$
if $\tau_j$-Cauchy, if and only $\langle \eta_{\tau_j}(\fF)\rangle$ 
converges to a point $s_j$ in the complete group $G_{\tau_j}$, if and only if $\fF$ contains
\[
(\tau_j)_{s_j}=\{\eta_{\tau_j}^{-1}(U):U\text{ neighbourhood of }s_j\text{ in }G_{\tau_j}\}.
\]

($\Rightarrow$)  Let $\fF$ be a $\tau_1$-Cauchy filter on $G$, and
let $\fF_j=\langle \eta_{\tau_j}(\fF)\rangle$ for $j=1,2$.  
Then $\fF_2$ contains $(\eta^{\tau_2}_{\tau_1})^{-1}(\eta_{\tau_1}((\tau_1)_{s_1}))$
for some $s_1$ in $G_{\tau_1}$, hence each element of $\fF_2$ has non-empty
intersection with the compact coset $C=(\eta^{\tau_2}_{\tau_1})^{-1}(\{s_1\})$.
Hence $\fF_2$ admits a cluster point $s_2$ in $C$, and $\langle \fF\cap(\tau_2)_{s_2}\rangle$
is a $\tau_2$-Cauchy refinement of $\fF$.

($\Leftarrow$)  Let us first see that $K=\ker\eta^{\tau_2}_{\tau_1}$ is compact.
Let $\fF_2$ be a filter on $K$, for which $\fF=\langle \eta_{\tau_2}^{-1}(\fF_2)\rangle\supseteq
(\tau_1)_{e_1}$, where $e_1$ is the identity of $G_{\tau_1}$, and hence is $\tau_1$-Cauchy.
But then $\fF$ admits a $\tau_2$-Cauchy refinement $\fF'$, so
$\fF'\supseteq(\tau_2)_{s_2}$ for some $s_2$ in $G_{\tau_2}$, which is necessarily
an element of the closed subgroup $K$.  

We now show that $\eta^{\tau_2}_{\tau_1}$ is open.
If $s_1\in G_{\tau_1}$, $\fF=\langle(\tau_1)_{s_1}\rangle$ admits a $\tau_2$-Cauchy refinement
$\fF'$, hence containing $ \langle(\tau_2)_{s_2}\rangle$ for some $s_2$ in $G_{\tau_2}$.
It is clear that $\eta^{\tau_2}_{\tau_1}(s_2)=s_1$, and hence $\eta^{\tau_2}_{\tau_1}:
G_{\tau_2}\to G_{\tau_1}$ is surjective, and induces a continuous group isomorphism
$\theta:G_{\tau_2}/K\to G_{\tau_1}$.  If, now, $\fF_1'$ is any filter in $G_{\tau_1}$
converging to $s_1$, then let $\fF'$ be a $\tau_2$-Cauchy refinement of 
$\langle\eta_{\tau_1}(\fF_1')\rangle$, so $\langle\eta_{\tau_2}(\fF')\rangle$ converges
to $s_2$, and $\theta(s_2K)=\eta^{\tau_2}_{\tau_1}(s_2)=s_1$. 
Hence $s_2K$ is a the unique cluster point of the filter $\theta^{-1}(\fF_1)$, 
thus the limit point.  This means that $\theta^{-1}$ is continuous, whence $\theta$ is open.
\end{proof}

\subsection{Central idempotents and the Galois connection}
Let $G=(G,\tau_G)$ be a weakly almost periodic group.  
The following result is crucial to our ensuing results.  
For the benefit of non-specialists,
we review its proof, indicating aspects from
the books \cite{burckel,ruppertB,berglundjm}.

\begin{theorem}\label{theo:ruppert1}
{\rm \cite[II.4.13 (iv)]{ruppertB}}  Any closed subsemigroup $S$ of $G^\fW$
admits a unique minimal idempotent, and a unique minimal ideal
which is a compact topological group.  
\end{theorem}

\begin{proof}
The semigroup $S$, being compact and semi-topological, admits both a minimal idempotent
$e$ and a unique minimal ideal $K(S)$ which must contain $e$
(\cite[Theo.\ 2.2]{burckel},\cite[I.2.1, I.2.6 \& I.3.5]{ruppertB} or
\cite[3.11]{berglundjm}).  Then $\idem(Se)$ is a right-zero semigroup, in particular $ee'=e'$ for
$e'$ in $\idem(Se)$; and $\idem(eS)$ is a left-zero semigroup.  It is a consequence of a 
structure theorem given in \cite[I.3.8]{ruppertB} or \cite[2.16 (ii)]{berglundjm}, that if
$\idem(Se)=\{e\}=\idem(eS)$, then $K(S)=eSe$ is a closed subgroup of $S$, and
hence admits only the idempotent $e$.

We shall find it essential to use the following form of the Ellis-Lawson
joint continuity result, as given in \cite[II.4.11]{ruppertB},
but with roles of left and right reversed and
applied to the compact left ideal $G^\fW x$ for some $x$ in $G^\fW$.  For nets
$(s_i),(t_i)$ in $G^\fW$ we have that 
\[
\lim_i(s_i,t_ix)=(s,x)\text{ in }G^\fW\times G^\fW \quad\Rightarrow\quad sx=\lim_i s_it_ix\text{ in }G^\fW.
\]

We shall prove that $\idem(Se)=\{e\}$, the argument for $\idem(eS)$ being symmetric.
 Fix a net $(t_i)\subset \eps^\fW(G)$ with $e=\lim_i t_i$.  We follow the argument of \cite[5.9]{berglundjm}.
Let $s_i=t_i^{-1}$, and, by passing to subnet, we may suppose that $s=\lim_is_i$ exists in $G^\fW$.
We now assume that $e'\in\idem(Se)$.  We then have that $\lim_i t_ie'=ee'=e$ so
\begin{align*}
e'&=e'e=\lim_i s_it_ie'e=\left(\lim_i s_i\right)\left(\lim_i t_ie'e\right) \\
&=se=\left(\lim_i s_i\right)\left(\lim_it_ie\right)=\lim_i s_it_ie=e.
\end{align*}

It is the classical application of the joint continuity theorem that $K(S)$ is a topological group.
\end{proof}

We consider the set of central idempotents:
\[
\ze(G^\fW)=\{e\in G^\fW:e^2=e\text{ and }es=se\text{ for all }s\text{ in }G^\fW\}.
\]
Then $\ze(G^\fW)$ is evidently a subsemigroup, in fact a subsemilattice of $G^\fW$.
We let $e_1\leq e_2$ in $\ze(G^\fW)$ provided $e_1e_2=e_1$.

The following result is shown in \cite{ruppert} for locally compact abelian groups.
Curiously, the main extra ingredient in the proof of our general result
is Theorem \ref{theo:ruppert1}.   Though we could borrow much
of the proof offered  in \cite{ruppert}, we give a complete proof for completeness of presentation
and further because we wish to extract more information, as will be demonstrated
in the corollaries.

\begin{theorem}\label{theo:galois}
There exist two maps
\[
T:\ze(G^\fW)\to\fT(G)\text{ and }E:\fT(G)\to \ze(G^\fW)
\]
which satisfy the relations
\begin{align*}
&T(e_1)\subseteq T(e_2)\text{ if }e_1\leq e_2\text{ in }\ze(G^\fW); \\
&E(\tau_1)\leq E(\tau_2)\text{ if }\tau_1\subseteq\tau_2\text{ in }\fT(G); \\
&E(\tau_1)=E(\tau_2)\text{ if }\tau_1\subseteq_c\tau_2\text{ in }\fT(G)\text{; and} \\
&E\circ T=\id_{\ze(G^\fW)}\text{ while }\tau\subseteq_c T\circ E(\tau).
\end{align*}
Hence if we let $\wbar{\fT}(G)=T(\ze(G^\fW))$, then $E:\wbar{\fT}(G)\to\ze(G^\fW)$ is an order-preserving
bijection with inverse $T$.  
\end{theorem}

\begin{proof}
(I) We first construct $T$.  Given $e$ in $\ze(G^\fW)$, we let
$\eta_e:G\to G^\fW$ be given by $\eta_e(s)=e\eps^\fW(s)$.  Then let
\[
T(e)=\sig(G,\{\eta_e\}).
\]
Notice that $\eta_e:G\to G^\fW(e)$, where the latter is topological group, so $T(e)$
is a group topology.  Furthermore, as $\fW^{T(e)}(G)=\fC(G^\fW)\circ\eta_e$,
we see that $T(e)=\sig(G,\fW^{T(e)}(G))$.  Thus $T(e)\in\fT(G)$.

We remark that 
\begin{equation}\label{eq:GTEisGWE}
G_{T(e)}=G^\fW(e).  
\end{equation}
Indeed, it suffices to show that
$e\eps^\fW(G)$ is dense in $G^\fW(e)$.  But this follows form the facts that
$\eta_e(G)=e\eps^\fW(G)\subseteq G^\fW(e)\subseteq eG^\fW$, and
$\eps^\fW(G)$ is dense in $G^\fW$.

(I') Observe that if $e_1\leq e_2$ in $\ze(G^\fW)$, then the map $\eta^{e_2}_{e_1}:G^\fW(e_2)
\to G^\fW(e_1)$ given by $\eta^{e_2}_{e_1}(s)=e_1s$, is continuous and satisfies
$\eta^{e_2}_{e_1}\circ\eta_{e_2}=\eta_{e_1}$, i.e.\ $\eta_{e_1}^{-1}(U)=\eta_{e_2}^{-1}
((\eta^{e_2}_{e_1})^{-1}(U))$ for open $U$ in $G^\fW(e_1)$.  Hence $T(e_1)\subseteq T(e_2)$.

(II)  We now construct $E$.  We let for $\tau$ in $\fT(G)$,
$\eta^\fW_\tau:G^\fW\to G_\tau^\fW$
be the unique continuous extension of $\eta_\tau$, given in (\ref{eq:tileta}).
Let $e^\fW_\tau=\eta^\fW_\tau\circ\eps^\fW(e_G)$, 
which is the identity of $G_\tau^\fW$, and $S_\tau=(\eta^\fW_\tau)^{-1}(\{e^\fW_\tau\})$,
so $S_\tau$ is a closed subsemigroup of $G^\fW$.  
Then Theorem \ref{theo:ruppert1} provides a unique minimal idempotent $E(\tau)$ for $S_\tau$.

We wish to verify that $E(\tau)\in\ze(G^\fW)$.  We let
$K_\tau=K(S_\tau)$ be the minimal ideal, which is a group thanks to Theorem \ref{theo:ruppert1}, 
with identity $E(\tau)$.
Given $s=\eps^\fW(t)$, for some $t$ in $G$, it is evident that $sS_\tau s^{-1}=S_\tau$.
This induces a semigroup automorphism on $S_\tau$ and hence
$sK_\tau s^{-1}$ is a minimal ideal, thus $sK_\tau s^{-1}=K_\tau$.
Furthermore, $sE(\tau)s^{-1}$ is an idempotent, hence the unique idempotent in
$K(\tau)$, so $sE(\tau)s^{-1}=E(\tau)$, or $sE(\tau)=E(\tau)s$.  By density of $\eps^\fW(G)$
in $G^\fW$, we see that $E(\tau)\in\ze(G^\fW)$.

Now if $e\in\idem(S_\tau)$, then $eE(\tau),E(\tau)e$ are idempotents dominated by $E(\tau)$, 
hence $E(\tau)$.  Hence $E(\tau)\leq e$.
We may thus express $E(\tau)$ by the succinct formula
\begin{equation}\label{eq:Etau}
E(\tau)=\min\idem(S_\tau)\text{ where }S_\tau=(\eta^\fW_\tau)^{-1}(\{e^\fW_\tau\}).
\end{equation}

(II') If $\tau_1\subseteq\tau_2$, then $(\eta^{\tau_2}_{\tau_1})^\fW\circ\eta^\fW_{\tau_2}
=\eta^\fW_{\tau_1}$, so $S_{\tau_1}\supseteq S_{\tau_2}$.  Thus $S_{\tau_2}K_{\tau_1}
\subseteq K_{\tau_1}$, so $E(\tau_2)E(\tau_1)$, being a product of commuting idempotents,
is an idempotent in $K_{\tau_1}$, hence is $E(\tau_1)$.  In other words, $E(\tau_1)\leq E(\tau_2)$.

If, further, $\tau_1\subseteq_c\tau_2$, then we have that for $K=\ker\eta^{\tau_2}_{\tau_1}$
there are isomorphisms
\begin{equation}\label{eq:aveK}
\fW(G_{\tau_1})\cong\fW(G_{\tau_2}/K)\cong\fW(G_{\tau_2})\ast m_K
\end{equation}
where each $f\ast m_K(s)=\int_K f(sk)\,dk$ ($s\in G_{\tau_2}$) is convolution with respect
to the normalized Haar measure on $K$.    Notice that $\eps^\fW_{\tau_2}(K)$ is a compact centric 
(i.e.\ $s\eps^\fW_{\tau_2}(K)=\eps^\fW_{\tau_2}(K)s$ for all $s$ in $G_{\tau_2}^\fW$) 
subgroup of the unit group of $G_{\tau_2}^\fW$. Then the algebras of (\ref{eq:aveK}) give 
character spaces
\[
G_{\tau_1}^\fW\cong G_{\tau_2}^\fW\ast m_{\eps^\fW_{\tau_2}(K)}
\cong G_{\tau_2}^\fW/\eps^\fW_{\tau_2}(K)
\]
where the last space is the semigroup $G_{\tau_2}^\fW$ modulo the congruence
$s\sim t$ if and only if $s\eps^\fW(K)=t\eps^\fW_{\tau_2}(K)$.  In particular, $s\sim e^\fW_{\tau_2}$
if and only if $s\in \eps^\fW_{\tau_2}(K)$.  Hence
if $e$ in $\ze(G^\fW)$ satisfies $\eta_{\tau_1}^\fW(e)=e^\fW_{\tau_1}$, then
since $\eta_{\tau_1}^\fW=(\eta^{\tau_2}_{\tau_1})^\fW\circ\eta_{\tau_2}^\fW$, we have
$\eta_{\tau_2}^\fW(e)\in ((\eta^{\tau_2}_{\tau_1})^\fW)^{-1}(\{e^\fW_{\tau_1}\})=
\eps^\fW(K)$, and thus $\eta_{\tau_2}^\fW(e)=e^\fW_{\tau_2}$.
Then, by (\ref{eq:Etau}) and the relation $E(\tau_1)\leq E(\tau_2)$, we see that $E(\tau_1)=E(\tau_2)$.

(III) We wish to verify the relations of the compositions, $E\circ T$ and $T\circ E$.
First, if $e\in\ze(G^\fW)$, then since $G_{T(e)}=G^\fW(e)\subseteq eG^\fW$, as given in 
(\ref{eq:GTEisGWE}), wee see  that
$\eta^\fW_{T(e)}(e')=e$ for any $e'$ in $\idem(G^\fW)$ for which $e'\geq e$.  Thus, using 
(\ref{eq:Etau}),   we obtain $E(T(e))=\min\idem((\eta^\fW_{T(e)})^{-1}(\{e\}))=e$.

Now if $\tau\in\fT(G)$,  (\ref{eq:Etau}) provides that $\eta^\fW_\tau(E(\tau))=e^\fW_\tau$, and 
we see for $s$ in $G$ that $\eta^\fW_\tau(\eps^\fW(s)E(\tau))=\eps^\fW_\tau\circ\eta_\tau(s)$, i.e.\ 
$\eps^\fW_\tau\circ\eta_\tau=\eta^\fW_\tau\circ\eta_{E(\tau)}\circ\eps^\fW$ on $G$, where $\eta_{E(\tau)}$
is defined as in (I).  Since $\eps^\fW_\tau:G_\tau\to G_\tau^\fW$ is a homeomorphism onto its range,
it follows that $\tau\subseteq T(E(\tau))$. Letting $K_\tau=K(S_\tau)$, as in (II), above, we see that
\begin{equation}\label{eq:EGWeqK}
E(\tau)G^\fW\cap S_\tau=E(\tau)S_\tau=K_\tau.
\end{equation}

Now let $\fF$ be a $\tau$-Cauchy filter on $G$.  We shall use Lemma \ref{lem:cocompact}
and the notation of its proof.  Then $\langle\eta_\tau(\fF)\rangle$ converges to 
an element $t$ in $G_\tau$.
Also, by compactness, $\langle\eta_{E(\tau)}(\fF)\rangle$ admits a cluster point $s$ in $E(\tau)G^\fW$.
We find that $\eta^\fW_\tau(s)=\eps^\fW_\tau(t)$.  Letting
$\fF^{-1}=\{\{r^{-1}:r\in F\}:F\in\fF\}$, we get Cauchy filter with, as above, associated
limit point $t^{-1}$ in $G_\tau$ and $s'$ in $E(\tau)G^\fW$.  But then
$ss'$ and $s's$ are in $K_\tau$.  Hence there are elements $l,r$ in $K_\tau$
for which $ls's=E(\tau)=ss'r$.  Thus $ls'=ls'E(\tau)=ls'ss'r=E(\tau)s'r=s'r$, and this element is the
inverse to $s$ in $E(\tau)G^\fW$, whence $s\in G^\fW(E(\tau))$.  Thus
we see that the filter $\langle \eta_{T(E(\tau))}(\fF)\rangle=\langle\eta_{E(\tau)}(\fF)\rangle$ 
on $G_{T(E(\tau))}=G^\fW(E(\tau))$ (see (\ref{eq:GTEisGWE}))
admits a Cauchy refinement, hence
$\fF$ admits a $T(E(\tau))$-Cauchy refinement on $G$.  Thus $\tau\subseteq_c T(E(\tau))$.
\end{proof}

Let us relate our theorem to the theory of ordered sets.  Our standard reference
is the treatise of Gierz {\it et al}, \cite{gierzhklms}.

\begin{remark}\label{rem:connection}
{\bf (i)} The pair of maps, $(E,T)$ is a  {\it Galois connection} or {\it adjunction} for the pair of partially
ordered sets $(\fT(G),\ze(G^\fW))$.  Indeed, this follows from the relations 
$E\circ T=\id_{\ze(G^\fW)}$ and $\tau\subseteq T\circ E(\tau)$, and \cite[O-3.6]{gierzhklms}.
In this terminology, $E$ is the {\it upper adjoint} and $T$ is the {\it lower adjoint}.
We shall use this in Section \ref{ssec:lattice} to establish the lattice structure on $\wbar{\fT}(G)$,
and in Section \ref{ssec:lc} to answer a question left open in the investigations \cite{ilies,ilies1}.

{\bf (ii)} The map $T\circ E:\fT(G)\to\wbar{\fT}(G)$, being order preserving
and idempotent, is a {\it closure operator}.  Whe will call
the set $\wbar{\fT}(G)=T\circ E(\fT(G))$, above, the set of {\it maximally cocompact closures} of elements of 
$\fT(G)$.   Observe that if $\kappa$ is any pre-compact topology, i.e.\ for which
$G_\kappa$ is compact, then $T(E(\kappa))=\tau_{ap}$,  where
$\tau_{ap}$ is defined in Example \ref{ex:ap}, since $G^{\fA\fP}$ is the largest
compact subgroup admitting a continuous dense image of $G$.  Hence
if $G$ is non-compact and  weakly almost periodic, then 
$\{\tau_{ap},\tau_G\}\subseteq\wbar{\fT}(G)$.

\end{remark}

Let us highlight some aspects of what we have shown above, which will be useful in the sequel.

\begin{corollary}\label{cor:galois1}
Given $\tau\in\fT(G)$,  the minimal ideal $K_\tau$ of the semigroup 
$S_\tau=(\eta^\fW_\tau)^{-1}(\{e^\fW_\tau\})$
is isomorphic to $\ker\eta^{T(E(\tau))}_\tau$.
\end{corollary}

\begin{proof}
It follows from (\ref{eq:EGWeqK}), and in the notation used there,
that $K_\tau=G^\fW(E(\tau))\cap S_\tau$.   Since $G_{T(E(\tau))}=G^\fW(E(\tau))$, 
as observed in (\ref{eq:GTEisGWE}), we have that 
$\eps^\fW\circ\eta^{T(E(\tau))}_\tau\cong\eta^\fW_\tau|_{G^\fW(E(\tau))}$.  It then follows that
$K_\tau\cong\ker\eta^{T(E(\tau))}_\tau$.
\end{proof}

\begin{corollary}\label{cor:galois2}
Given $\tau\in\fT(G)$ we have topological isomorphisms of semigroups:

{\bf (i)}  $G^\fW_\tau\cong E(\tau)G^\fW$ if $\tau\in\wbar{\fT}(G)$; and

{\bf (ii)} $G^\fW_\tau\cong(E(\tau)G^\fW)/K_\tau$ if $\tau\in\fT(G)$.
\end{corollary}

\begin{proof}
We have that $\eta^\fW_\tau(E(\tau))=e^\fW_\tau$.  Letting
$\eta_{E(\tau)}:G^\fW\to G^\fW$ be the map of multiplication by $E(\tau)$ we then have
that
\begin{equation}\label{eq:epsWtisEepsW}
\eps^\fW_\tau\circ\eta_\tau=\eta^\fW_\tau\circ\eps^\fW=\eta^\fW_\tau\circ\eta_{E(\tau)}\circ\eps^\fW
\text{ on }G.
\end{equation}
Hence the map $\eta^\fW_\tau$ factors through $\eta_{E(\tau)}(G^\fW)=E(\tau)G^\fW$.
Conversely, since $G_\tau\cong G^\fW(E(\tau))$, as follows from (\ref{eq:GTEisGWE}), then 
the universal property of $G^\fW_\tau$ shows that $E(\tau)G^\fW$ must be a quotient of
$G^\fW_\tau$.  Thus we obtain (i).   Then (ii) follows from  part (II')
of the proof of Theorem \ref{theo:galois}.
\end{proof}

\subsection{Lattice structure}\label{ssec:lattice}
Notice that $\ze(G^\fW)$ is a semilattice with $e\wedge e'=ee'$, possessing 
a maximal element $e^\fW_G=\eps^\fW(e_G)$, and a minimal element $e_{ap}=E(\tau_{ap})$.  
If $S\subseteq\ze(G^\fW))$ is non-empty then we let its infimum be given by
\begin{equation}\label{eq:infzegw}
\prod S=\lim_{\substack{F\Subset S \\ F\nearrow S}}\prod_{e\in F}e
\end{equation}
where $F\Subset S$ indicates that $F$ is a non-empty finite subset of $S$ and the family
of such set is directed by inclusion. 
Let us say a few words to verify that this infimum is defined. 
Consider the net $(e_F)_F=(\prod_{e\in F}e)_F$.  
For any pair $(e_{F(j)})_j$ and $(e_{F(i)})_i$ of converging subnets, 
we have $e_Fe_{F(j)}=e_{F(j)}$ for $F(j)\supseteq F$ and similar for $i$, so letting $s=\lim_j e_{F(j)}$ 
and $t=\lim_i e_{F(i)}$ we have
\[
s=\lim_i\lim_je_{F(i)}e_{F(j)}=st=\lim_j\lim_i e_{F(i)}e_{F(j)}=t.
\]
Hence $(e_F)_F$ admits a unique cluster point, hence a limit which is clearly an idempotent,
which is easily verified to be central.  Then $\lim_Fe_F$ is clearly a lower bound for $S$.
If $e$ in $\ze(G^\fW)$ is any other lower bound, then $ee_F=e$ for any $F$, so $e\lim_Fe_F=e$,
i.e.\ $e\leq \lim_Fe_F$.  With similar effort, this result may be deduced from \cite[I.2.6]{ruppertB}.
We will simply define the supremum by
\[
\bigvee S=\bigwedge\{e\in\ze(G^\fW):es=s\text{ for all }s\text{ in }S\}.
\]
We do not presently
have a more ``intrinsic" definition of supremum.  

If $\fS\subseteq\fT(G)$ is non-empty then its intersection $\bigcap\fS$ is certainly a group topology
on $G$.  Indeed, one needs only to inspect the inverse image of $U$ in $\bigcap\fS$ under
both the product and inversion maps.
Moreover, $\bigcap\fS\in\fT(G)$ as we have that 
\[
\bigcap\fS=\bigcap_{\tau\in\fS}\sig(G,\fW^\tau(G))=\sig\left(G,\bigcap_{\tau\in\fS}\fW^\tau(G)\right)
=\sig(G,\fW^{\bigcap\fS}(G))
\]
The definition of the supremum is more satisfactory than for $\ze(G^\fW)$, above.
We let 
\[
\bigvee\fS=\sig(G,\{\del\})\text{ where }\del:G\to\prod_{\tau\in\fS}G_\tau,\;\del(t)=(\eta_\tau(t))_{\tau\in\fS}.
\]
We observe that since we have an embedding into a compact semi-topological semigroup,
$\prod_{\tau\in\fS}G_\tau\hookrightarrow\prod_{\tau\in\fS}G_\tau^\fW$, we obtain that
$\bigvee\fS\in\fT(G)$.

It is not a priori obvious for $\tau_1,\tau_2$ in $\wbar{\fT}(G)$ that either of $\tau_1\vee\tau_2$ or
$\tau_1\cap\tau_2$ is in $\wbar{\fT}(G)$.  However, this follows from the completeness of the lattice
$\fT(G)$, as is evident from the next result.

\begin{proposition}\label{prop:tarski}
The subset $\wbar{\fT}(G)$ of maximal cocompact closures of $\fT(G)$ is a complete
sublattice of $\fT(G)$.
\end{proposition}

\begin{proof}
This is an application of Tarski's fixed point theorem, \cite[O-2.3]{gierzhklms}.
Indeed,  $\wbar{\fT}(G)$ is the set of fixed points
for the monotone closure operator $T\circ E$ on $\fT(G)$.  
\end{proof}

\begin{proposition}\label{prop:lattice}
The maps $E:\fT(G)\to\ze(G^\fW)$ and $T:\ze(G^\fW)\to\wbar{\fT}(G)\subseteq\fT(G)$ 
enjoy the following lattice continuity properties:
\[
E(\bigcap\fS)=\prod E(\fS)\text{ and }E(\bigvee\fS)=\bigvee E(\fS)
\]
for any non-empty subset $\fS$ of $\fT(G)$, respectively in $\wbar{\fT}(G)$, and
\[
T(\bigvee S)=\bigvee T(S)\text{ and }T(\prod S)=\bigcap T(S)
\]
for any non-empty subset $S$ of $\ze(G^\fW)$. 
\end{proposition}

\begin{proof}
As noted in Remark \ref{rem:connection}, 
The pair $(E,T)$ is the pair of upper and lower adjunctions for the pair
$(\fT(G),\ze(G^\fW))$, while $(T,E)$ is the pair of upper and lower adjuctions for the pair
$(\ze(G^\fW),\wbar{\fT}(G))$.  Upper adjuctions respect infema, while lower adjunctions respect
suprema, thanks to \cite[O-3.3]{gierzhklms}.
\end{proof}

\section{Eberlein-de Leeuw-Glicksberg ideals}

\subsection{Definition and main result}
Let $G=(G,\tau_G)$ be a weakly almost periodic group.

\begin{definition}
An {\it Eberlein-de Leeuw-Glicksberg ideal}  (E-dL-G ideal) of $\fW(G)$ is a proper closed proper ideal
$\fJ$ which satisfies

$\bullet$ $\fJ$ is translation invariant, and

$\bullet$ there is a C*-subalgebra $\fA$ of $\fW(G)$ which is a closed linear complement of $\fI$.
\end{definition}

The above definition allows $\{0\}$ to be considered a E-dL-G ideal.

If $s\in G^\fW$ and $u\in\fW(G)$, we let for $t$ in $G$ $u\cdot s(t)=\lim_i u(s_it)$
where $(s_i)$ is any net for which $s=\lim_i\eps^\fW(s_i)$.

The following is our main result regarding these ideals.

\begin{theorem}\label{theo:main}
{\bf (i)}  Given $\tau$ in $\wbar{\fT}(G)$, $\fW^\tau(G)=\fW(G)\cdot E(\tau)$.  Moreover
\[
\fI(\tau)=\{u\cdot E(\tau)=0:u\in\fW(G)\}
\]
is a E-dL-G ideal  for which $\fW^\tau(G)$ is its linear complement. Hence
\[
\fW(G)=\fW^\tau(G)\oplus\fI(\tau).
\] 

{\bf (ii)} Any E-dL-G ideal $\fJ$ is of the form $\fI(\tau)$ of (i), above.
\end{theorem}

\begin{proof}
(i)  That $\fW^\tau(G)=\fW(G)\cdot E(\tau)$ is a simple consequence of
Corollary \ref{cor:galois2} (i). 
Moreover, the set $\fI(\tau)$ is the set of elements of $\fW(G)\cong\fC(G^\fW)$ which
vanish on $E(\tau)G^\fW$, and hence $\fI(\tau)$ is an ideal in $\fW(G)$.
Since $E(\tau)G^\fW$ is an ideal in $G^\fW$, $\fI(\tau)$ is translation invariant.

(ii) Let $H=\{s\in G^\fW:s(\fJ)=0\}$ be the hull of $\fJ$.  Since $\fJ$ is translation-invariant
we find that $\eps^\fW(G)\cdot H,H\cdot\eps^\fW(G)\subseteq H$, and it follows from
density and continuity that $H$ is an ideal in $G^\fW$.
Let $\fA$ be a C*-algebra which is the linear complement of $\fJ$. 
Then the composition of maps $\fA\hookrightarrow \fW(G)\to\fW(G)/\fJ$
must be an injective $*$-homomorphism, hence isomorphism of C*-algebras, and,
moreover, commutes with the actions of translation.
We conclude that $\fA$ is unital and translation invariant.
Thus the Gelfand spectrum $G^\fA$ of $\fA$ is topologically isomorphic to
the semigroup $H$.  In particular, the element $e$ in $H$ corresponding to
$\eps^\fA(e_G)$ in $G^\fA$ is the identity for $H$.  

If $t\in \eps^\fW(G)$, then (\ref{eq:tileta}) then $\gamma_t(s)=tst^{-1}$ defines
a continuous homomorphism on $G^\fW$ with inverse $\gamma_{s^{-1}}^\fW$,
hence an automorphism of $G^\fW$.  Since $H$ is an ideal, it follows that
$\gamma_s^\fW|_H$ is an automorphism on $H$, so $\gamma_s^\fW(e)=e$.
In other words, $e$ commutes
with each element of $\eps^\fW(G)$, and thus $e\in \ze(G^\fW)$.  

We conclude that
$H=eG^\fW$, and hence $\fJ=\{u\in\fW(G):u\cdot e=0\}=\fI(T(e))$,
since $e=E(T(e))$, by Theorem \ref{theo:galois}.
\end{proof}

For the case of $\tau_{ap}$, we recover exactly the classical decomposition (\ref{eq:meandecomp}),
due to Eberlein, and to de Leeuw and Glicksberg.

Since for $\tau$ in $\wbar{\fT}(G)$, $\fW^\tau(G)$ is the subalgebra of $\tau$-continuous
elements of $\fW(G)$, we may deem $\fI(\tau)$ to be the ideal of ``purely $\tau$-discontinous
functions" in $\fW(G)$.  This ideal enjoys the following property.

\begin{lemma}\label{lem:bergelsonr}
If $\tau$ in $\wbar{\fT}(G)$, and $u_1,\dots,u_n\in\fI(\tau)$, then for any $\eps>0$
and for any neighbourhood $U$ of $e_G$ in $\tau$, there is $t$ in $U$ for which
$|u_k(t)|<\eps$ for each $k$.
\end{lemma}

In other words, for each finite subset $\fF$ of $\fI(\tau)$, each $\tau$-neighbourhood
of $e_G$ contains ``nearly vanishing" elements of $G$ for $\fF$.

\begin{proof}
From part (I) of the proof of Theorem \ref{theo:galois} we see that
$U=\eta_{E(\tau)}^{-1}(V)=\{t\in G:\eps^\fW(t)E(\tau)\in V\}$ for some neighbourhood
$V$ of $E(\tau)$ in $G^\fW(E(\tau))\cong G_\tau$.  Now fix $s$ in $V$ and let
$(t_i)$ be a net in $G$ for which $\lim_i\eps^\fW(t_i)=s$ in $G^\fW$.  But then
$\lim_i\eps^\fW(t_i)E(\tau)=sE(\tau)=s$ and it follows that $\eps^\fW(t_i)E(\tau)$ is eventually
in $V$, and $t_i$ is eventually in $U$.  It also follows that
for each $k$ we have $\lim_iu_k(t_i)=\lim_i\hat{u}_k(\eps^\fW(t_i))=\hat{u}_k(s)=0$,
where $\hat{u}_k$ in $\fC(G^\fW)$ is the Gelfand transform of $u_k$.
\end{proof}

If $\tau\in\fT(G)\setminus\wbar{\fT}(G)$, then 
it follows from (\ref{eq:aveK}) and Corollary \ref{cor:galois1} that we get 
a linear decomposition
\[
\fW^{T(E(\tau))}(G)=\fW^\tau(G)\oplus\left\{u\in\fW^{T(E(\tau))}(G):\int_{K_\tau}u\cdot k\,dk=0\right\}
\]
where the second summand is merely a closed, translation-invariant linear space, and not generally an
algebra.  We do see, however, that $u\mapsto \int_{K_\tau}u\cdot k\,dk$ is an expectation from
$\fW^{T(E(\tau))}(G)$ onto $\fW^\tau(G)$.


\subsection{Decompositions of weakly almost periodic representations}
Let $\fX$ be a Banach space.
A weak operator continuous representation $\pi:G\to\fB(\fX)$
is {\it weakly almost periodic} if the weak operator closure $\wbar{\pi(G)}^{wo}$
in $\fB(\fX)$ is weak operator compact.  Hence  the weak operator closure of
the convex hull of $\pi(G)$, $\Aff_\pi^{wo}(G)$, is also weak operator compact, thanks to 
the Krein-\u{S}mulian theorem.

\begin{theorem}\label{theo:contdecomp}
Let $\pi:G\to\fB(\fX)$ be a unital weakly almost periodic representation
and $\tau\in\fT(G)$, and denote
\[
\fX_\tau^\pi=\{\xi\in\fX:\pi(\cdot)\xi\text{ is }\tau\text{-continuous on }G\}.
\]
Then $\fX_\tau^\pi$ is a $\pi(G)$-reducing subspace, i.e.\ there is an idempotent
$M_\tau^\pi\in\Aff_\pi^{wo}(G)$ for which $\pi(s)M_\tau^\pi=M_\tau^\pi\pi(s)$ for $s$ in $G$ and
$\fX_\tau^\pi=M_\tau^\pi\fX$.

If, moreover, $\tau\in\wbar{\fT}(G)$, then $M_\tau^\pi\in\wbar{\pi(G)}^{wo}$, and
the $\pi$-invariant complement $\fX_{\tau\perp}^\pi=(I-M_\tau^\pi)\fX$
of $\fX_\tau^\pi$ is described by
\[
\fX_{\tau\perp}^\pi=\left\{\xi\in\fX:\text{ for any }\tau\text{-neighbourhood }U\text{ of } e_G,\;
0\in\wbar{\pi(U)\xi}^w\right\}.
\]
\end{theorem}

\begin{proof}
We let $\pi^\fW:G^\fW\to\wbar{\pi(G)}^{wo}$ be the unique continuous extension of $\pi$.
We then let $M_\tau^\pi=\int_{K_\tau}\pi^\fW(k)\,dk$, where the integral is understood
in a weak operator sense.  Hence $M_\tau^\pi\in \Aff_\pi^{wo}(G)$ via the usual approximation
of probability measures by convex combination of point masses form a dense subset.  That
$M_\tau^\pi$ commutes with elements of $\pi(G)$ follows from centricity of $K_\tau$
in $G^\fW$.

If $\tau\in\wbar{\fT}(G)$, then $K_\tau=\{E(\tau)\}$, which shows that
$M_\tau^\pi\in\wbar{\pi(G)}^{wo}=\pi^\fW(G^\fW)$.
Further, if $\xi\in \fX_{\tau\perp}^\pi$, then for any $\eta$ in $\fX^*$, then
$0=\eta(\pi(\cdot)M_\tau^\pi\xi)=\eta(\pi(\cdot)\xi)\cdot E(\tau)$, so
$\eta(\pi(\cdot)\xi)\in\fI(\tau)$.  Hence it follows Lemma \ref{lem:bergelsonr}
that $0\in\wbar{\pi(U)\xi}^w=\wbar{\pi(U)}^{wo}\xi$.
Converesely, if $\xi\in\fX\setminus\fX_{\tau\perp}^\pi$, then we may find
$\eta\in\fX^*$ for which $1=\eta(M_\tau^\pi\xi)=\eta(M_\tau^\pi\pi(e_G)\xi)$.
Then $(M_\tau^\pi)^*\eta(\pi(\cdot)\xi)=
\eta(M_\tau^\pi\pi(\cdot)\xi)=\eta(\pi(\cdot)\xi)\cdot E(\tau)$ is a non-zero element of $\fW^\tau(G)$
and hence for any $\eps<1$, there is a $\tau$-neighbourhood of $e_G$
on which $|(M_\tau^\pi)^*\eta(\pi(\cdot)\xi)|\geq \eps$.
\end{proof}

\begin{remark}
We refer to the subspace $\fX_{\tau\perp}^\pi$ as the space of {\it $\tau$-weakly mixing}
vectors of $\pi$.  Hence if $\fX$ is a Hilbert space, then
the $\tau_{ap}$-weakly mixing vectors are the ``weakly mixing"
vectors of Bergelson and Rosenblatt \cite{bergelsonr} or of Dye \cite{dye}, or 
``flight" vectors of Jacobs \cite{jacobs}.   Notice that in this case,  we may first
specify a $\tau_{ap}$ neighbourhood which contains our nearly vanishing element of $G$.
Hence we gain a mild refinement of Bergelson and Rosenblatt's description of these vectors.
\end{remark}

We also gain a partial generalization of many of the decomposition results from \cite{deleeuwg1}, which
in turn generalizes a result in \cite{vonneumanns}.  This is shown by simply observing that $\tau_G\in
\fT(G_d)$ where $G_d$ is the discretisation of $G$, and applying Theorem \ref{theo:contdecomp}.

\begin{corollary}\label{cor:vvnsdlg}
Let $(G,\tau_G)$ be a weakly almost periodic group, and $\pi:G_d\to\fB(\fX)$ an almost periodic 
representation.  Then the space of continuous vectors, $\fX^\pi_{\tau_G}$, is a $\pi(G)$-reducing subspace. 
In particular $\fW(G)$ is a complemented translation-invariant subspace of $\fW(G_d)$.
\end{corollary}

\begin{remark}\label{rem:dlg}
{\bf (i)} Note that in \cite{deleeuwg1}, many theorems are stated for {\em locally almost periodic representations}, 
i.e.\ $\pi:G\to\fB(\fX)$ for which $\pi(V)$ is relatively weak operator compact for some open $V$.  Since 
our methods are global, they do not point to any means of obtaining decompositions for locally almost 
periodic representations.

We remark that for a locally compact group $(G,\tau_G)$, every locally bounded
weak operator continuous representation on a Banach space is automatically strong operator continuous,
thanks to \cite[Theo.\ 2.8]{deleeuwg1}.  In this case, the space $\fX^\pi_{\tau_G}$ of Corollary
\ref{cor:vvnsdlg} consists of strongly continuous vectors.

{\bf (ii)}  We do not generally expect $\fW(G)$ to be the complement of an E-dL-G ideal
in $\fW(G_d)$.  For example, on $\Ree_d$ we have
$\tau_\Ree\subsetneq\bar{\tau}_\Ree=\tau_\Ree\vee\tau_{ap}$.  See
\cite{ilies,ilies1} and Proposition \ref{prop:cocomloccom}, below.
\end{remark}

As a further application, we consider the quotient via a closed normal subgroup.  If $(G,\tau_G)$ is
a weakly almost periodic group and $N$ is a closed normal subgroup, we let
$q_N:G\to G/N$ denote the quotient map and
\begin{equation}\label{eq:quottop}
\tau_{G:N}=\sig(G,\fW(G/N)\circ q_N).
\end{equation}
Each element of $\fW(G/N)\circ q_N$ is constant on cosets of $N$.  Likewise if
$u$ in $\fW(G)$ is constant on costs of $N$, $\til{u}(sN)=u(s)$ is well defined on $G/N$
with $\til{u}\cdot (G/N)=u\cdot G$ relatively weakly compact. Hence
$\fW^{\tau_{G:N}}(G)=\fW(G/N)\circ q_N$ is simply the subalgebra of elements which are constant on
cosets of $N$.  The following is now an evident consequence of Theorem \ref{theo:contdecomp}.

\begin{corollary}\label{cor:laul}
If $(G,\tau_G)$ is a weakly almost periodic group with closed normal subgroup $N$, and
$\pi:G\to\fB(\fX)$ is a weakly almost periodic representation, then
\[
\fX^\pi_{\tau_{G:N}}=\{\xi\in\fX:\pi(n)\xi=\xi\text{ for each }n\text{ in }N\}
\]
is a $\pi(G)$-reducing subspace.  In particular $\fW(G/N)\circ q_N$ is a complemented
translation invariant subspace of $\fW(G)$.
\end{corollary}

\begin{remark}\label{rem:laul}
{\bf (i)} If $G$ is locally compact, the quotient by any closed normal subgroup is also locally compact.
In this case, the corollary above is obtained in \cite{laul}.  

{\bf (ii)} Returning to weakly almost periodic groups,
notice that $\tau_{G:N}$ induces a topology $\tau_{G;N}$ on $G/N$.  
We shall see in Section \ref{sec:comparison} that $\tau_{G;N}$ may strictly coarser
than the quotient topology $\tau_{G/N}$ on $G/N$.  To the author's knowledge,
it is an open question as to whether $\tau_{G;N}$ is Hausdorff on $G/N$.
\end{remark}

\subsection{Lattice stucture of E-dL-G ideals}
We first observe the order-theoretic properties of our realisation of E-dL-G ideals.

\begin{proposition}\label{prop:antitone}
The map $\fI:\wbar{\fT}(G)\to\edlg(G)$, i.e.\ from the maximal cococpact weakly almost periodic topologies
to the set of E-dL-G ideals of $\fW(G)$, is antitone:  
\[
\fI(\tau_1)\supseteq\fI(\tau_2)\text{ if }\tau_1\subseteq\tau_2.  
\]
Furthermore, $(\fI,\fI^{-1})$ is an antitone Galois connection
on the pair $(\wbar{\fT}(G),\edlg(G))$.
\end{proposition}

\begin{proof}
If $\tau_1\subseteq\tau_2$ in $\wbar{\fT}(G)$, then $E(\tau_1)E(\tau_2)=E(\tau_2)$.  Hence it follows
form definition that $\fI(\tau_1)\supseteq\fI(\tau_2)$.  We observe that $\fI$ is bijective,
thanks to Theorem \ref{theo:main}.
\end{proof}

We gain the following join operations  on E-dL-G ideals.

\begin{proposition}\label{prop:sumideals}
Let $\fS\subset\wbar{\fT}(G)$ be non-empty.  If $\fS$ is finite, then
\[
\fI(\bigcap\fS)=\sum_{\tau\in\fS}\fI(\tau).
\]
In particular the latter set is closed.  For general such $\fS$ we have 
\[
\fI(\bigcap\fS)=\overline{\sum_{\tau\in\fS}\fI(\tau)}.
\]
\end{proposition}

\begin{proof}
For the finite case, we will show this only for $\fS=\{\tau_1,\tau_2\}$, since the general finite case follows 
from induction.
We have that $E(\tau_1\cap\tau_2)=E(\tau_1)E(\tau_2)$, by Proposition \ref{prop:lattice}.
If $u_j\in\fI(\tau_j)$ for $j=1,2$ we have 
$(u_1+u_2)\cdot E(\tau_1\cap\tau_2)=(u_1\cdot E(\tau_1))\cdot E(\tau_2)+
(u_2\cdot E(\tau_2))\cdot E(\tau_1)=0$, so $\fI(\tau_1)+\fI(\tau_2)\subseteq\fI(\tau_1\cap\tau_2)$.
Conversely, if $u\in\fI(\tau_1\cap\tau_2)$, write $u=u-u\cdot E(\tau_1)+u\cdot E(\tau_1)$.
Then $u-u\cdot E(\tau_1)\in\fI(\tau_1)$ and $u\cdot E(\tau_1)\in\fI(\tau_2)$, since
$E(\tau_1)E(\tau_2)=E(\tau_1\cap\tau_2)$.

To see the general case we first see that if $u\in\fW^{\bigcap\fS}(G)$
then we may use Proposition \ref{prop:lattice}, (\ref{eq:infzegw}) and 
to see that
\[
u=u\cdot E(\bigcap \fS)=u\cdot\prod E(\fS)
=\lim_{\substack{F\Subset S \\ F\nearrow S}} u\cdot \prod E(\fF)
\]
where each $u\cdot E(\fF)\in\sum_{\tau\in\fF}\fI(\tau)$, and convergence is pointwise
in $\fC(G^\fW)\cong\fW(G)$.  On the relatively weakly compact set $u\cdot G$ the pointwise
and weak topologies coincide (see \cite[A.6]{berglundjm}, for example), and hence by an application
of Hahn-Banach theorem we see that $u\in\overline{\sum_{\tau\in\fS}\fI(\tau)}$.
To obtain the converse inclusion we observe that
\[
\sum_{\tau\in\fS}\fI(\tau)=\bigcup_{\fF\Subset \fS}\sum_{\tau\in\fF}\fI(\tau)
=\bigcup_{\fF\Subset \fS}\fI(\bigcap \fF).  
\]
Each $\fI(\bigcap \fF)\subseteq\fI(\bigcap \fS)$, by Proposition \ref{prop:antitone}. \end{proof}

\begin{remark}\label{rem:EdLGmeet}
We notice that if $\tau_1,\tau_2\in\wbar{\fT}(G)$, then $\fI(\tau_1)\cap\fI(\tau_2)$ may be identified
as the ideal in $\fC(G^\fW)$ of elements vanishing on the ideal $E(\tau_1)G^\fW\cup E(\tau_2)G^\fW$,
which is generally smaller than the ideal $E(\tau_1\vee\tau_2)G^\fW$.  
Hence the infimum $\fI(\tau_1)\wedge\fI(\tau_2)$, qua E-dL-G ideal, does not generally 
coincide with usual meet $\fI(\tau_1)\cap\fI(\tau_2)$.  See Example \ref{ex:eeucmot}, below.
This seems to be in keeping with the fact that we 
do not have an intrinsic description of $E(\tau_1)\vee E(\tau_2)=E(\tau_1\vee\tau_2)$ in terms of
$E(\tau_1)$ and $E(\tau_2)$.  
\end{remark}

With the considerations above, for a non-empty subset $\fS$ of $\wbar{\fT}(G)$  we have lattice
relations
\[
\bigwedge_{\tau\in\fS}\fI(\tau)=\fI(\bigvee\fS)\text{ and }
\bigvee_{\tau\in\fS}\fI(\tau)=\fI(\bigcap\fS)
\]
where the latter may be understood as the usual supremum of ideals, thanks to 
Proposition \ref{prop:sumideals}.

\section{Locally precompact topologies}

\subsection{Definition and semilattice structure}\label{ssec:lc}
Let $(G,\tau_G)$ be a topological group.
We define the set of {\it pre-locally compact topologies} on $G$ by
\[
\fT_{lc}(G)=\left\{\sig(G,\{\eta\}):\begin{matrix} \eta:G\to H\text{ is a continuous homomor-} \\ 
\text{phism into a locally compact group }H\end{matrix}\right\}.
\]
If $\tau=\sig(G,\{\eta\})$, as above, then $G_\tau=\wbar{\eta(G)}\subset H$ is, up to topological isomorphism, 
the unique completion with respect to the two-sided uniformity (or left, or right uniformity)
with respect to $\tau$.  See the discussion in \cite[Sec.\ 1]{ilies}.  Hence if
$\tau\in\fT_{lc}(G)$, we let $\eta_\tau:G\to G_\tau$ denote the map into the completion.

The inclusions $\fC_0(G_\tau)\circ\eta_\tau\subseteq\fW(G_\tau)\circ\eta_\tau=\fW^\tau(G)$
show that $\fT_{lc}(G)\subseteq\fT(G)$.  

The initial topology on $G$ generated by the map
$s\mapsto (\eta_{\tau_1}(s),\eta_{\tau_2}(s))\in G_{\tau_1}\times G_{\tau_2}$ is $\tau_1\vee\tau_2$.
Hence $(\fT_{lc}(G),\vee)$ is a subsemilattice of $\fT(G)$. 

The following is suggested by  \cite[Thm.\ 2.2]{ilies}. 
Regrettably, the proof of the result, as stated there, is wrong, but corrected in
the Corrigendum \cite{ilies1}.  Part (ii), however, answers a question left open.

\begin{proposition}\label{prop:cocomloccom}
{\bf (i)} $\wbar{\fT}_{lc}(G)=T\circ E(\fT_{lc}(G))\subseteq \fT_{lc}(G)$.
Hence $\wbar{\fT}_{lc}(G)=\fT_{lc}(G)\cap\wbar{\fT}(G)$.

{\bf (ii)} If $\tau_1,\tau_2\in\wbar{\fT}_{lc}(G)$, then $\tau_1\vee\tau_2\in\wbar{\fT}_{lc}(G)$.
Thus $(\wbar{\fT}_{lc}(G),\vee)$ is a semilattice.
\end{proposition}

\begin{proof}  
Since $G_\tau\cong G_{T(E(\tau))}/\ker\eta^{T(E(\tau))}_\tau$ is locally compact with
$K=\ker\eta^{T(E(\tau))}_\tau$ compact, $H=G_{T(E(\tau))}$ is locally compact.
Indeed, let $\eta=\eta^{T(E(\tau))}_\tau:H\to G_\tau$, $V$ a compact neighbourhood of the identity in
$G_\tau$ and $W=\eta^{-1}(V)$.  If $\fF$ is a filter on $W$, then $\eta(\fF)$ is a filter on $V$, hence
with cluster point $t$.  It follows that every neighbourhood of the compact set $\eta^{-1}(\{t\})$ meets each
element of $\fF$.  Since $\wbar{F_1\cap F_2}\subseteq\wbar{F}_1\cap\wbar{F}_2$, we see that
$\{\wbar{F}\cap\eta^{-1}(\{t\}):F\in\fF\}$ is a filter base on $\eta^{-1}(\{t\})$, and hence admits a cluster point,
which in turn must be a cluster point for $\fF$.  Thus $W$ is compact, and we have (i).

Part (ii) is immediate from part (i) and Proposition \ref{prop:tarski}.
\end{proof}

We do not generally have that $\wbar{\fT}_{lc}(G)$ is a complete semilattice (see \cite[Sec.\ 6.3]{ilies}), 
and no indication of whether or not it is generally a lattice.

We note the following dichotomy result.

\begin{proposition}\label{prop:dichotomy}
Let $\tau\in\fT(G)$.  Then either

{\bf (i)} $\tau\in\fT_{lc}(G)$, which is exactly the case in which
$G^\fW(E(\tau))$ is open in $E(\tau)G^\fW$; or

{\bf (ii)} $E(\tau)G^\fW\setminus G^\fW(E(\tau))$ is dense in $E(\tau)G^\fW$.
\end{proposition}

\begin{proof}
If $\tau\in\fT_{lc}(G)$, then so too is $T(E(\tau))$, by Proposition \ref{prop:cocomloccom}.
If $T(E(\tau))\in \fT_{lc}(G)$, then soo too is $G_\tau/\ker\eta^{T(E(\tau))}_\tau$.
Hence let us assume that $\tau=T(E(\tau))\in\wbar{\fT}_{lc}(G)$.
In this case $\fC_0(G_\tau)\subseteq \fW(G_\tau)$, from which it follows that
$G_\tau^\fW\setminus \eps^\fW(G_\tau)$ is closed in $G_\tau^\fW$.  Thanks to
Corollary \ref{cor:galois2} and (\ref{eq:GTEisGWE}), this is equivalent to having
that $G^\fW(E(\tau))$ is open in $E(\tau)G^\fW$.  This is case (i).

If $G^\fW(E(\tau))$ in $E(\tau)G^\fW$, contains a neighbourhood of any of its points,
then it must be open.  Hence if $G^\fW(E(\tau))$ is not open, then we get (ii).
\end{proof}

\subsection{The spine algebra}
In this section we give a topological construction of the spine algebra.  A harmonic analysis
based construction is the main topic of \cite{ilies,ilies1}.  For abelian groups, this construction, in a 
dual context, goes back to \cite{varopoulos,taylor,inoue}.

Given $\tau\in\fT_{lc}(G)$, let $\fC_0^\tau(G)=\fC_0(G_\tau)\circ\eta_\tau$.  We define the
{\em spine algebra} by
\[
\fC_0^*(G)=\wbar{\sum_{\tau\in\fT_{lc}(G)}\fC_0^\tau(G)}\subseteq \fC_\beta(G).
\]
Let us see that this is indeed an algebra.  If $\tau_1,\tau_2\in\fT_{lc}(G)$, then
multiplication on $\fC_0^{\tau_1}(G)\otimes \fC_0^{\tau_2}(G)$ factors through
the injective tensor product
\[
\fC_0^{\tau_1}(G)\check{\otimes} \fC_0^{\tau_2}(G)
\cong \fC_0(G_{\tau_1})\check{\otimes} \fC_0(G_{\tau_2})
\cong \fC_0(G_{\tau_1}\times G_{\tau_2})
\]
as the map of restricting to the diagonal, i.e.\ if $u_j\in\fC_0(G_{\tau_j})$, $j=1,2$, then for $s$
in $G$ we have $u_1\circ\eta_{\tau_1}\cdot u_2\circ\eta_{\tau_2}(s)=
u_1\otimes u_2(\eta_{\tau_1}(s),\eta_{\tau_2}(s))$.  This defines an element of
$\fC_0^{\tau_1\vee\tau_2}(G)$.  Comparing with the join on $\fT_{lc}(G)$, we see that
\begin{equation}\label{eq:comult}
\fC_0^{\tau_1}(G)\fC_0^{\tau_2}(G)\subseteq \fC_0^{\tau_1\vee\tau_2}(G).
\end{equation}

In order to learn about the structure of the spine algebra, we shall use the following.

\begin{lemma} \label{lem:laul}
{\rm \cite[Lem.\ 12]{laul}} If $H$ is a locally compact group, then all non-zero translation invariant
C*-subalgebras of $\fC_0(H)$ are of the form $\fC_0(H/K)\circ q_K$, where
$K$ is a compact normal subgroup and $q_K:H\to H/K$ is the quotient map.
\end{lemma}

An an immediate consequence we see that for $\tau_1,\tau_2$ in $\fT_{lc}(G)$ that
\begin{equation}\label{eq:cocontain}
\fC_0^{\tau_1}(G)\subseteq \fC_0^{\tau_2}(G)\quad\Leftrightarrow\quad
\tau_1\subseteq_c\tau_2.
\end{equation}
Furthermore, we see that 
\begin{align}\label{eq:cointersect}
\fC_0^{\tau_1}(G)&\cap \fC_0^{\tau_2}(G)\not=\{0\} \quad\Leftrightarrow \\
&T\circ E(\tau_1)=T\circ E(\tau_2)=T\circ E(\tau_1\vee\tau_2). \notag
\end{align}
Indeed, if non-zero, the intersection algebra is isomorphic to the algebra
$\fC_0(G_{\tau_1}/K_1)\cong\fC_0(G_{\tau_2}/K_2)$ for compact subgroups $K_j\subseteq G_{\tau_j}$,
$j=1,2$, and hence corresponds to $\tau_3\subseteq_c\tau_j$.  But then by properties of the closure 
operator $T\circ E$ we have $T(E(\tau_3))=T(E(\tau_j))$, for $j=1,2$.  Hence we find, using Proposition
\ref{prop:lattice} that each $T\circ E(\tau_j)=T\circ E(\tau_1)\vee T\circ E(\tau_2)=T\circ E(\tau_1\vee\tau_2)$.
Conversely, the latter condition of (\ref{eq:cointersect}) gives that each $\tau_j\subseteq_c\tau_1\vee\tau_2$,
and the intersection is naturally isomorphic to
$\fC_0(G_{\tau_1\vee\tau_2}/( \ker\eta^{\tau_1\vee\tau_2}_{\tau_1}\cdot\ker\eta^{\tau_1\vee\tau_2}_{\tau_1}))$.

Combining (\ref{eq:cocontain}) and (\ref{eq:cointersect}) we  have
\[
\fC_0^*(G)=\,\,\,u\text{-}\!\!\!\bigoplus_{\tau\in\wbar{\fT}_{lc}(G)}\fC_0^\tau(G)
\]
where the direct sum $u\text{-}\bigoplus$ indicates that the direct sum is completed in the uniform norm.
The multiplication between the factors is indicated by (\ref{eq:comult}).  Hence this is
a commutative C*-algebra, graded over the semilattice $(\wbar{\fT}_{lc}(G),\vee)$.
Similar examples are considered in \cite{mageira}.

The spectrum of $\fC_0^*(G)$ is a certain {\em spine compactification} $G^*$ of $G$.
This object is described in \cite{ilies}, and a description of its topology given there, though with
a correction in \cite{ilies1}.  This is a Clifford semigroup, of which specialized versions of this may be seen in 
\cite{berglund} and \cite{mageira}.  Let us note that
\[
G^*=\bigsqcup_{\tau\in\what{\fT}_{lc}(G)}G_\tau
\]
where 
\begin{equation}\label{eq:hatTlc}
\what{\fT}_{lc}(G)=\{\bigvee\fS:\fS\subseteq \wbar{\fT}_{lc}(G)\} 
\end{equation}
is the lattice completion
of $\wbar{\fT}_{lc}(G)$.  We note that for each $\fS\subseteq \wbar{\fT}_{lc}(G)$, the set
$\fS'=\{\tau\in \wbar{\fT}_{lc}(G):\tau\subseteq\bigvee\fS\}$ is both hereditary and directed, and the group
$G_{\bigvee\fS}\cong G_{\fS'}$, where the latter notation is that used in \cite{ilies}.  
If $\tau_1,\tau_2\in\wbar{\fT}_{lc}$, then $\tau_1\wedge\tau_2=\bigvee\{\tau\in\wbar{\fT}_{lc}(G):
\tau\subseteq\tau_1\cap\tau_2\}$.  Then the multiplication on $G^*$
is given as follows:  if $s\in G_{\tau_1}$ and $t\in G_{\tau_2}$, then
\[
st=\eta^{\tau_1}_{\tau_1\wedge\tau_2}(s)\eta^{\tau_2}_{\tau_1\wedge\tau_2}(t)\in G_{\tau_1\wedge\tau_2}
\]
where $\eta^\tau_\tau=\id_{G_\tau}$.

\subsection{On totally minimal groups}
A locally compact group $(G,\tau_G)$ is called {\it totally minimal} if each quotient
by a closed normal subgroup admits a unique Hausdorff group topology.
Connected totally minimal groups are characterized by Meyer~\cite{mayer1}.

For totally minimal locally compact $G$ we let $\fN(G)$ denote the lattice of closed normal subgroups
and it is evident that we have
\[
\fT(G)=\{\tau_{G:S}:S\in\fN(G)\}=\fT_{lc}(G)
\]
where $\tau_{G:S}$ is defined as in (\ref{eq:quottop}).  We define a {\it topological 
commensurability relation} on closed normal subgroups:  $S\sim S'$ if
both $S/(S\cap S')$ and $S'/(S\cap S')$ are compact.  We note that
\[
\underline{S}=\bigcap\{S':S'\sim S\}
\]
is closed and normal.  The quotient $S/\underline{S}$ is a closed subgroup
of $G/\underline{S}$, which admits the unique topology arising from the ``diagonal" embedding
\[
x\underline{S}\mapsto (x(S\cap S'))_{\substack{S'\in\fN(G) \\ S'\sim S}}:
G/\underline{S}\to \prod_{\substack{S'\in\fN(G) \\ S'\sim S}}G/(S'\cap S)
\] 
and it follows that $S/\underline{S}$ is compact.
Thus each equivalence class admits a minimal element and we let $\underline{\fN}(G)$ denote the set
of these minimal representatives.  We see, then, that
\[
\wbar{\fT}(G)=\{\tau_{G:S}:S\in\underline{\fN}(G)\}=\wbar{\fT}_{lc}(G).
\]
We note that $\tau_G=\tau_{G:\{e\}}$ and $\tau_{ap}=\tau_{G:C}$, where $C$ is the smallest
co-compact normal subgroup in $G$.  

Furthermore, for $S,S'$ in $\fN(G)$ we have
\begin{equation}\label{eq:latTGmin}
\tau_{G:S}\vee\tau_{G:S'}=\tau_{G:S\cap S'}\text{ and }
\tau_{G:S}\cap\tau_{G:S'}=\tau_{G:\wbar{SS'}}.
\end{equation}
The pair of the map $S\mapsto \tau_{G:S}$ and its inverse are an antitone Galois connection for the pair
$(\fN(G),\fT(G))$, and hence for the pair $(\underline{\fN}(G),\wbar{\fT}(G))$, as follows Tarski's
fixed point theorem (as referenced in Proposition \ref{prop:tarski}).

When $G$ is connected, it follows from \cite[Theo.\ 4.5]{mayer1} that
$G$ is totally minimal exactly when
\[
\fW(G)=\,\,u\text{-}\!\!\!\bigoplus_{S\in\underline{\fN}(G)}\fC_0(G/S)\circ q_S
=\,\,u\text{-}\!\!\!\bigoplus_{S\in\underline{\fN}(G)}\fC_0^{\tau_{G:S}}(G)=\fC^*_0(G).
\]
We observe, then, that for $S$ in $\underline{\fN}(G)$ we have
\[
\fW^{\tau_{G:S}}(G)=\,\,u\text{-}\!\!\!\bigoplus_{\substack{S'\in\underline{\fN}(G) \\ S'\supseteq S}}\fC_0^{\tau_{G:S}}(G)
\text{ and }
\fI(\tau_{G:S})=\,\,u\text{-}\!\!\!\bigoplus_{\substack{S'\in\underline{\fN}(G) \\ S'\not\supset S}}\fC_0^{\tau_{G:S}}(G).
\]

\begin{example}\label{ex:eeucmot}
The group $G=\Cee^2\rtimes\Tee$, with multiplication 
\[
(z,w,t)(z',w',t')=(z+tz',w+tw',tt')
\]
is minimal,
as indicated in \cite[Rem.\ 18]{mayer2}.  Let $\fL$ denote the family of one-dimensional $\Cee$-subspace of 
$\Cee^2$. Then
\[
\fN(G)=\underline{\fN}(G)=\{\{e\},L\rtimes\{1\}\,(L\in\fL),\Cee^2\rtimes\{1\}\}.
\]

If $L\not=L'$ in $\fL$ then (\ref{eq:latTGmin}) provides
$\tau_{G:L\rtimes\{1\}}\vee \tau_{G:L'\rtimes\{1\}}=\tau_G$.
Hence we conclude for the meet of E-dL-G ideals that
\[
\fI(\tau_{G:L\rtimes\{1\}})\wedge\fI(\tau_{G:L'\rtimes\{1\}})=\{0\}
\] 
while
\[
\fI(\tau_{G:L\rtimes\{1\}})\cap\fI(\tau_{G:L'\rtimes\{1\}})=
\left(\,\,u\text{-}\!\!\!\!\!\bigoplus_{M\in\fL\setminus\{L,L'\}}\fC_0^{\tau_{G:M\rtimes\{1\}}}(G)\right)
\oplus_u\fC_0(G)
\]
as suggested in Remark \ref{rem:EdLGmeet}.
\end{example}

\section{Unitarizable topologies}

\subsection{Definition and properties}
Given a group $G$ and a group topology $\tau$ on $G$ we let $\pd^\tau(G)$ denote the set
of $\tau$-continuous positive definite functions.  We let the {\it unitarizable topologies}
on a topological group $(G,\tau_G)$ be given by
\[
\wbar{\fT}_u(G)=\{\tau\subseteq\tau_G:\tau\text{ is a group topology, with }\tau=\sig(G,\pd^\tau(G))\}.
\]
Let us justify the term ``unitarizable".  Given $u$ in $\pd^\tau(G)$ we review the Gelfand-Naimark-Segal
construction on $u$.  On $\Cee[G]$ we consider the positive Hermitian form
\[
\left[\sum_{i=1}^n\alp_is_i,\sum_{j=1}^m\beta_jt_j\right]_u=\sum_{i=1}^n\sum_{j=1}^m\alp_i\wbar{\beta_j}u(t_j^{-1}s_i)
\]
and let $\fK_u=\{\xi\in\Cee[G]:[\xi,\xi]_u=0\}$.  We induce a form, hence a Euclidean norm on $\Cee[G]/\fK_u$, and
let $\fH_u$ be the completion of the latter space.  We define $\pi_u:G\to \un(\fH_u)$ by extending
$\pi_u(s)\xi=\sum_{i=1}^n\alp_iss_i+\fK_u$, where $\xi=\sum_{i=1}^n\alp_is_i+\fK_u$.  Then for
$\eta=\sum_{j=1}^m\beta_jt_j+\fK_u$ we have $\langle\pi(s)\xi,\eta\rangle
=\sum_{i=1}^n\sum_{j=1}^m\alp_i\wbar{\beta_j}u(t_j^{-1}ss_i)$.  Hence it is evident that
$\pi_u$ is $\tau$-weak-operator continuous.
Let $\xi_u=e+\fK_u$, so $u=\langle \pi_u(\cdot)\xi_u,\xi_u\rangle$.  We then let
\[
\vpi_\tau=\bigoplus_{u\in\pd^\tau(G)}\pi_u\text{ on }
\fH_\tau=\,\,\ell^2\text{-}\!\!\!\bigoplus_{u\in\pd^\tau(G)}\fH^\pi_\tau
\]
and it is straightforward to check, using the polarization identity, (\ref{eq:polarization}) below, that
\[
\sig(G,\{\vpi_\tau\})=\sig(G,\pd^\tau(G)).
\]
Hence $\vpi_\tau:(G,\tau)\to\un(\fH_\tau)$ is a homeomorphism onto its range.  We use the well-known fact 
that $so=wo$ on $\un(\fH_\tau)$, as is immediate from the identity
\begin{equation}\label{eq:woso}
\|u\xi-v\xi\|^2=2\|\xi\|^2-2\mathrm{Re}\langle u\xi,v\xi\rangle
\end{equation}
for $u,v$ in $\un(\fH_\tau)$, $\xi$ in $\fH_\tau$.

If $\tau_G\in\fT_u(G)$, then we call $(G,\tau_G)$ a {\it unitarizable group}.  For such a group we let
$\vpi=\vpi_{\tau_G}$ on $\fH_G=\fH_{\tau_G}$
be the {\it universal unitary representation}.  We shall assume that $(G,\tau_G)$ is unitarizable
for the rest of this section, and write $\pd(G)=\pd^{\tau_G}(G)$.  
We let $G^\vpi=\wbar{\vpi(G)}^{wo}$ (weak operator closure) which is a compact semi-topological
semigroup.  Notice, then, that $\pd(G)\subset\fC(G^\vpi)\circ\vpi\subseteq\fW(G)$, so
$\fT_u(G)\subseteq\fT(G)$.

We note that $\fT_u(G)$ is a complete sublattice of $\fT(G)$.  Indeed, let
$\fS\subseteq\fT_u(G)$.  Then
\[
\vpi_\fS=\bigoplus_{\tau\in\fS}\vpi_\tau\text{ on }\fH_\fS=\ell^2\text{-}\bigoplus_{\tau\in\fS}\fH_\tau
\]
satisfies $\bigvee\fS=\sig(G,\{\vpi_\fS\})\in\fT_u(G)$.    Just as with weakly almost periodic
topologies we see that
\[
\bigcap\fS=\bigcap_{\tau\in\fS}\sig(G,\pd^\tau(G))=\sig(G,\pd^{\bigcap\fS}(G))
\]
so $\bigcap\fS\in\fT_u(G)$ too.

 We consider central idempotents
\[
\ze(G^\vpi)=\{p\in G^\vpi:p^2=p\text{ and }\vpi(s)p=p\vpi(s)\text{ for each }s\text{ in }G\}.
\]
Since $G^\vpi$ is contained in the weak-operator compact unit ball of bounded operators $\fB(\fH_G)$,
$\ze(G^\vpi)$ consists of projections: $p=p^2=p^*$.

\begin{theorem}\label{theo:galoisU}
There exist two maps
\[
T_u:\ze(G^\vpi)\to\fT_u(G)\text{ and }P:\fT_u(G)\to\ze(G^\vpi)
\]
which satisfy the relations
\begin{align*}
&T_u(p_1)\subseteq T(p_2)\text{ if }p_1\leq p_2\text{ in }\ze(G^\vpi); \\
&P(\tau_1)\leq P(\tau_2)\text{ if }\tau_1\subseteq\tau_2\text{ in }\fT_u(G); \\
&P(\tau_1)=E(\tau_2)\text{ if }\tau_1\subseteq_c\tau_2\text{ in }\fT_u(G)\text{; and} \\
&P\circ T_u=\id_{\ze(G^\vpi)}\text{ while }\tau\subseteq_c T_u\circ P(\tau).
\end{align*}
Hence if we let $\wbar{\fT}_u(G)=T_u(\ze(G^\vpi))$, then $P:\wbar{\fT}_u(G)\to\ze(G^\fW)$ is an order-preserving
bijection with inverse $T$.  

Furthermore, $\wbar{\fT}_u(G)$ is a complete sublattice of $\fT_u(G)$.
\end{theorem}

\begin{proof}
It is a straightforward exercise to recreate the proof of Theorem \ref{theo:galois} in this context.
We remark on some details.
For the construction of $P$, we need to observe that $\vpi_\tau:G\to G^{\vpi_\tau}=
\wbar{\vpi_\tau(G)}^{wo}$ admits a unique continuous extension $\til{\vpi}_\tau:G^\vpi\to
G^{\vpi_\tau}$.  This is evident as $\vpi_\tau$ is a subrepresentation of $\vpi$.  
The intrinsic groups $G_\tau=G^{\vpi_\tau}(I)$ (where $I=\vpi_\tau(e_G)$)
are groups of unitaries, complete with respect to the two-sided uniformity
thanks to Theorem \ref{theo:ruppert}.  [In fact, a more elementary proof is available.
If $(v_i)$ is a Cauchy net in $G^{\vpi_\tau}(I)$, let $v$ in $G^{\vpi_\tau}$ be a 
cluster point.  Use (\ref{eq:woso}) to see that $v*v=I=vv^*$ and that this cluster point is unique.]
If $\tau_1\subseteq\tau_2$
in $\fT_u(G)$, we let $\eta^{\tau_2}_{\tau_1}:G_{\tau_2}\to G_{\tau_1}$ be the unique
continuous homomorphism extending the induced map from $G/\wbar{\{e_G\}}^{\tau_2}$ to
 $G/\wbar{\{e_G\}}^{\tau_1}$.  If
$\tau_1\subseteq_c\tau_2$ in $\fT_u(G)$, then $K=\ker\eta^{\tau_2}_{\tau_1}$, then
averaging over $K$ in $G^{\vpi_{\tau_2}}$ is achieved by compression by the central
projection $\int_K \vpi_{\tau_2}(k)\,dk$.

That $\wbar{\fT}_u(G)$ is a complete sublattice of $\fT_u(G)$ follows from the fact that 
$(P,T_u)$ is a Galois connection for the pair of partially ordered sets $(\fT_u(G),\ze(G^\vpi))$,
and Proposition \ref{prop:tarski}.
\end{proof}

The following is proved similarly as Corollary \ref{cor:galois2}.

\begin{corollary}\label{cor:galoisU1}
Let for $\tau$ in $\wbar{\fT}_u(G)$, $G^{\vpi_\tau}=\wbar{\vpi_\tau(G)}^{wo}$.
Then we have an isomorphism of semitopological semigroups
\[
G^{\vpi_\tau}=P(\tau)G^\vpi.
\]
\end{corollary}

It can be checked for $\tau$ in $\fT_u(G)$ that $P(\tau)=\vpi^\fW(E(\tau))$.  However, 
the following suggests that Theorem \ref{theo:galoisU} may not be thus witnessed as a straightforward
corollary of Theorem \ref{theo:galois}.

\begin{question}
If $\tau\in\fT_u(G)$ and $\tau'\in\fT(G)$ with $\tau\subseteq_c\tau'$, does it follow that $\tau'\in\fT_u(G)$?
\end{question}

If the answer to the question is `yes', then $\wbar{\fT}_u(G)$ is a sublattice of $\wbar{\fT}(G)$.
We note that this is already a quotient lattice.  Indeed, we observe that
$\ze(G^\vpi)$ is a quotient of $\ze(G^\fW)$.  If $p\in \ze(G^\vpi)$, we let
$S_p=(\vpi^\fW)^{-1}(\{p\})$, which by part (II) of the proof of Theroem \ref{theo:galois}, admits
a central idempotent, $e$, so $\vpi^\fW(e)=p$.  We identify $p=P(\tau)$ for some $\tau\in\wbar{\fT}_u(G)$,
and hence $e=E(\tau)$.

\subsection{Decompositions of Fourier-Stieltjes algebras}
We let the {\it Fourier-Stieltjes algebra} of a unitarizable group $(G,\tau_G)$ be given by
\[
\fsal(G)=\spn\pd(G).
\]
Thanks to polarization identity, for a continuous unitary representation $\pi:G\to\un(\fH)$ and $\xi,\eta$ in $\fH$
we have
\begin{equation}\label{eq:polarization}
4\langle \pi(\cdot)\xi,\eta\rangle=\sum_{k=0}^3 i^k\langle \pi(\cdot)(\xi+i^k\eta),\xi+i^k\eta\rangle
\end{equation}
and we see that $\fsal(G)$ consists of matrix coefficients of continuous unitary representations.
As with the case of a locally compact group, this is a subalgebra:  we use direct sums
and tensor products of continuous unitary representations, as in \cite{eymard}.
This is a Banach algebra, with a certain norm, as noted in \cite{laulu}.
Let us briefly indicate this, but with an alternative approach.  Many of the indicated results
are noted in \cite{laulu1}.

We let $\wstar(G)=\vpi(G)''$ denote the von Neumann algebra generated by
$\vpi(G)$ in $\fB(\fH_G)$, and $\wstar(G)\bar{\otimes}\wstar(G)$ its normal spatial tensor product.

\begin{theorem}\label{theo:fsal}
\begin{itemize}
\item[(i)] Each continuous unitary representation $\pi:G\to\un(\fH)$ admits a unique normal extension
$\pi'':\wstar(G)\to\pi(G)''$.
\item[(ii)] We have duality relationship $\fsal(G)^*=\wstar(G)$ with duality
\[
\langle u, T\rangle = \langle \pi''(T)\xi,\eta\rangle\text{ where }u=\langle\pi(\cdot)\xi,\eta\rangle.
\]
\item[(iii)] The map $\Gamma=(\vpi\otimes\vpi)'':\wstar(G)\to\wstar(G)\bar{\otimes}\wstar(G)$
admits preadjoint $\Gamma_*:\fsal(G)\hat{\otimes}\fsal(G)\to\fsal(G)$, which shows that
$\fsal(G)$, as the predual of $\wstar(G)$, is a Banach algebra of functions on $G$.
\item[(iv)]  Any closed translation-invariant subspace $\fal$ of $\fsal(G)$ is of the form
\[
\fal=Z\cdot \fsal(G)
\]
where $Z\in\zp(\wstar(G))$ is a central projection and 
\[
Z\cdot\langle\pi(\cdot)\xi,\eta\rangle
=\langle\pi(\cdot)Z\xi,\eta\rangle.
\]
\end{itemize}
\end{theorem}

The symbol $\hat{\otimes}$, above, indicates the {\it operator projective tensor product}, of this
predual space, $\fsal(G)=\wstar(G)_*$, with itself.  See \cite{effrosr} for more on this.
Hence we conclude that $\fsal(G)$ is a {\it completely contractive  Banach algebra}, which implies that
it is a Banach algebra.

\begin{proof}
By construction $\vpi$ is the direct sum of all continuous cyclic unitary representations of $G$.
Hence (i) is a simple matter of observing that $\pi$ can be gotten from $\vpi$ by taking
compressions, then amplifications.  The action of $\fsal(G)$ on $\wstar(G)$ as functionals, indicated in (ii),
is one of normal functionals, and must be separating as $\spn\vpi(G)$ is weak operator dense
in $\wstar(G)$, hence weak*-dense.  Each normal functional is a linear combination of $4$
positive normal functionals, each which has the indicated form.

Part (iii) is an immediate application of (i) and (ii) and duality.  

Part (iv) is \cite[(3.17)]{arsac} (in the case that $G$ is locally compact).  The present case
can be achieved with the same proof.
\end{proof}

Now we get to our main application.

\begin{theorem}\label{theo:EdLGFS}
Let $\tau\in\wbar{\fT}_u(G)$, and
\[
\fsal^\tau(G)=P(\tau)\cdot \fsal(G)\text{ and }\ideal(\tau)=(I-P(\tau))\cdot\fsal(G).
\]
Then
\begin{align*}
&\fsal^\tau(G)=\{u\in\fsal(G):u\text{ is }\tau\text{-continuous}\} \\
&\ideal(\tau)=\{\langle\pi(\cdot)\xi,\eta\rangle:\pi\text{ admits no }\tau\text{-continuous subrepresentations}\}.
\end{align*}
Moreover we have a decomposition
\[
\fsal(G)=\fsal^\tau(G)\oplus_{\ell^1}\ideal(\tau)
\]
into a closed translation-invariant subalgebra and a closed translation invariant ideal.
\end{theorem}

\begin{proof}
Let $\fE(G)=\wbar{\fsal(G)}^{\|\cdot\|_\infty}$, the uniform closure of $\fsal(G)$.
We call this the {\it Eberlein algebra} of $G$.
Then  $G^\vpi$ is the Gelfand spectrum of $\fE(G)$;
see, for example, \cite[Theo.\ 2.4]{spronks} or \cite[Theo.\ 3.12]{megrelishvili3}.  Analogously to Theorem \ref{theo:main},
we get a E-dL-G decomposition
\[
\fE(G)=\fE^\tau(G)\oplus\fI_\fE(\tau).
\]
Using Corollary \ref{cor:galoisU1} in place of Corollary \ref{cor:galois2} we see that
$\fE^\tau(G)=\fE(G)\cdot P(\tau)$ is the algebra of $\tau$-continuous Eberlein functions,
and $\fI_\fE(\tau)$ those Eberlein functions vanishing on $P(\tau)G^\vpi$.

We let $\fsal^\tau(G)=\fsal(G)\cap \fE^\tau(G)$.  We note that for $s$ in $G$, that
left translation and predual action by $\vpi(s)$ on $u$ in $\fsal(G)$ coincide:
$s\cdot u=\varpi(s)\cdot u$.  Hence letting $P(\tau)=wo\text{-}\lim_i\vpi(s_i)$ we see that
$P(\tau)\cdot u=u\cdot P(\tau)$ may be regarded, as an element of $\fE(G)$, in a sense analagous
to the action of $G^\fW$ on $\fW(G)$.  The formulas describing $\fsal^\tau(G)$ and $\ideal(\tau)$
follow immediately.

That the direct sum of the last formula is an $\ell^1$-direct sum follows from \cite[(3.9)]{arsac}.
\end{proof}

\begin{remark} Let $\Delta=\Delta(\fsal(G))\subset \wstar(G)$ denote the Gelfand spectrum of $\fsal(G)$.
We have that $G^\vpi\subseteq\Delta$.  Proper inclusion may hold, as is shown by the Wiener-Pitt phenomenon
on classical abelian groups.  If there were $Z\in\zp(\wstar(G))\cap\Delta\setminus G^\vpi$, then
$u\mapsto Z\cdot u$ is a homomorphism on $\fsal(G)$, with closed range, and hence
it kernel $(I-Z)\cdot\fsal(G)$ is an ideal.  We do not know of an alternative description of such ideals,
or even if such $Z$ exist.
\end{remark}

\section{Comparison with other classes of topologies}\label{sec:comparison}

We point to some theorems and examples in the literature, showing the scope of weakly almost periodic topologies.
In this section $(G,\tau_G)$ will be a Hausdorff topological group.
We note the following inclusion of families topologies which are defined in this paper:
\begin{equation}\label{eq:inclusions}
\fT_{lc}(G)\subseteq\fT_u(G)\subseteq\fT(G).
\end{equation}
To see the first inclusion, we have for $\tau$ in $\fT_{lc}(G)$ that
$\tau=\sig(G,\fC_0^\tau(G))=\sig(G,\{\lambda_\tau\circ\eta_\tau\})$, where
$\lambda_\tau:G_\tau\to\un(\bl^2(G_\tau))$ is the left regular representation of the locally compact
completion.

\subsection{Linear representability}  It is shown in \cite{teleman} that there is a weakly continuous representation
$\pi:G\to\mathrm{Is}(\fX)$  (isometries on $\fX$ with weak operator topology) with $\tau_G=\sig(G,\{\pi\})$.

\subsection{Reflexive representability}  {\it $(G,\tau_G)$ is weakly almost periodic if and only if
there is a reflexive Banach space $\fX$ and a weak operator continuous representation $\pi:G\to\mathrm{Is}(\fX)$ 
for which $\tau_G=\sig(G,\{(\pi,wo)\})$}, i.e.\ the topology on $G$ is recoverable from the weak operator
topology on $\mathrm{Is}(\fX)$.   This important result is in \cite{shtern}, and is improved 
in \cite{megrelishvili}, where it is further shown that $\tau_G=\sig(G,\{(\pi,so)\})$, i.e.\ strong 
and weak operator topologies must coincide on $\pi(G)$.  Compare with \cite[Theo.\ 2.8]{deleeuwg1}, as
mentioned in Remark \ref{rem:dlg}.

\subsection{Some non-weakly-almost-periodic groups}  \label{ssec:notwap}
{\bf (i)} The group $G=H_+[0,1]$ of orientation preserving homeomorphisms,
with compact-open topology,
has that $\fW(G)=\Cee 1$, as shown by \cite{megrelishvili1}.  Hence $\fT(G)=\bigl\{\{\varnothing,G\}\bigr\}$.

{\bf (ii)} The (real) Banach space $c_0$, with its norm topology
has that $\eps^\fW:c_0\to c_0^\fW$ is not a homeomorphism onto its range, as shown by \cite{ferrig}.  
This Banach space and other non-weakly-almost-periodic Banach spaces are indicated in \cite{benyaacovbf}.

\subsection{Some weakly almost periodic but non-unitarizable groups}
{\bf (i)} For $n$ in $\En$, the (real) Banach space $G_{2n}=\bl^{2n}(\mu)$
is shown to be reflexively representable, hence weakly almost periodic, in \cite{megrelishvili2}.
For $n\geq 2$, $G_{2n}$ is not unitarizable, as noted below.

{\bf (ii)} Following work of \cite{banaszczyk} and (i), above, it is shown in
\cite{glasnerw} that there is a discrete subgroup $\Gamma$ of the real Banach space
$\ell^4$, for which $\ell^1/\Gamma$ is weakly almost periodic, non-unitarizable, and monothetic.
Hence $\fT_u(\Zee)\subsetneq\fT(\Zee)$.

\subsection{Some non-locally compact unitarizable groups}\label{ssec:uspenskeij}
{\bf (i)}
Given a measure space $(X,\mu)$ the (real) Banach space $G_p=\bl^p(\mu)$ with its norm topology is unitarizable
exactly when $1\leq p\leq 2$.  This is the work of many authors, and the result put
together in \cite[Theo.\ 2.2]{galindo}.  The proof of sufficiency is curiously simple.  
We see that $\tau_{G_p}=\sig(G_p,\{e^{-{\|\bullet-v\|_p}^2}:v\in G_p\})$, where each
$e^{-{\|\bullet-v\|_p}^2}\in\pd(G_p)$, by an old result of Schoenberg.

{\bf (ii)}
For any completely regular Hausdorff space $X$, the universal (real) locally convex space
$G=L(X)$ generated by $X$ is unitarizable, and hence so too is the universal free abelian group 
$A(X)$ as is shown in \cite{uspenskeij}.
In the course of the proof, he also notes that the Banach space $\ell^1$ is unitarizable.

\subsection{On quotients of weakly almost periodic and unitarizable groups}
Combining resuls \ref{ssec:notwap} (ii) and \ref{ssec:uspenskeij} gives rise to
an interesting conculsion, noted in \cite{ferrig}.
Let $Q:\ell^1\to c_0$ be a linear quotient map, and $K=\ker Q$.  Then the Banach space
$\ell^1$ is unitarizable while $c_0=\ell^1/K$ is not.  In particular, in the notation 
of Remark \ref{rem:laul} (ii),  $\tau_{\ell^1;K}$ is not the quotient topology on $\ell^1/K$.

In particular, {\it the quotient of a unitarizable group need not be weakly almost periodic, in particular
need not be unitarizable.}

The author is extremely grateful to J. Galindo for indicating this example.

\subsection{Proper inclusions of families of topologies}  Let $G=\Zee^{\oplus\En}$, which
is the free abelian group of countably infinite rank.  If $\fX$ is a real Banach space with
dense spanning sequence $\{\xi_1,\xi_2,\dots\}$  (e.g.\ a Schauder basis), we define
$\eta_\fX:G\to\fX$ by
\[
\eta_\fX\left((n_j)_{j\in\En}\right)=\sum_{j=1}^\infty (n_{2j-1}\sqrt{2}+n_{2j})\xi_j
\]
and $\tau^\fX=\sig(G,\{\eta_\fX\})$.  It is a well-known fact that $\Zee\sqrt{2}+\Zee$ is dense
in $\Ree$, from which it follows that $(G,\tau^\fX)$ admits completion $G_{\tau^\fX}=\fX$.
This completion is, in fact, the metrical completion of $\eta_\fX(G)$ in the norm of $\fX$.
The examples $\tau_G$ (discrete topology), $\tau^{\ell^1}$, $\tau^{\bl^4(\mu)}$ ($\mu$ is Lebesgue
measure on $[0,1]$, say) and $\tau^{c_0}$ show that
\[
\fT_{lc}(G)\subsetneq\fT_u(G)\subsetneq\fT(G)\subsetneq\{\text{all group topologies on }G\}.
\]

We let $\what{\fT}_{lc}(G)$ be as given in (\ref{eq:hatTlc}).  Let, as above, $\fX$ be a separable
Banach space with fixed map $\eta_\fX$ and topology $\tau^\fX$ as above, and assume that $\tau^\fX\in\fT_u(G)$.
We shall see that 
\begin{equation}\label{eq:nonincl}
\bar{\tau}^\fX=T_u\circ P(\tau^\fX)\in\wbar{\fT}_u(G)\setminus \what{\fT}_{lc}(G).
\end{equation}
We first note that $\fX$ is an almost periodic group, since $\{e^{-if(\bullet)}:f\in\fX^*\}$ separates points.
Hence, as observed in \cite{ilies1}, $\bar{\tau}^\fX=\tau^\fX\vee\tau_{ap}$.
We let $\fS=\{\tau\in\wbar{\fT}_{lc}(G):\tau\subseteq\bar{\tau}^\fX\}$.  Hence $\bigvee\fS\in\what{\fT}_{lc}(G)$,
and is the maximal such element dominated by $\bar{\tau}^\fX$.  We then
consider the map $\eta^{\bar{\tau}^\fX}_{\bigvee\fS}:G_{\bar{\tau}^\fX}\to G_{\bigvee\fS}$, and 
let $\sig=\sig(G_{\bar{\tau}^\fX},\{\eta^{\bar{\tau}^\fX}_{\bigvee\fS}\})$.  Since $G_{\bar{\tau}^\fX}$ is 
an extension of $\fX$ be a compact group, we see for each $\tau$ in $\fS$ that $G_\tau$,
having dense subgroup $\eta^{\bar{\tau}^\fX}_\tau(G_{\bar{\tau}^\fX})$, is almost connected,
hence compactly generated.  Thus we see that $G_\tau=\Ree^{n_\tau}\times T$ where $n_\tau\in\En$
and $T$ is compact, and we have that $\tau=\tau_{n_\tau}\vee\tau_{ap}$, where $\tau_{n_\tau}$ is 
induced by the map from $G_{\bar{\tau}^\fX}$ into $\Ree^{n_\tau}$.  
For a net $(\xi_i)$  and a point $\xi_0$ in $G_{\bar{\tau}^\fX}$ we have
\begin{align*}
\sig\text{-}\lim_i \xi_i=\xi_0\;\Leftrightarrow\; 
&\lim_i\eta^{\bar{\tau}^\fX}_\tau(\xi_i)=\eta^{\bar{\tau}^\fX}_\tau(\xi_0) \text{ in }G_\tau\text{ for each }
\tau\text{ in }\fS \\
\Leftrightarrow\;&\lim_i\eta^{\bar{\tau}^\fX}_{\tau_{n_\tau}}(\xi_i)=\eta^{\bar{\tau}^\fX}_{\tau_{n_\tau}}(\xi_0)
\text{ in }\Ree^{n_\tau}
\text{ for each }\tau\text{ in }\fS \\
&\text{and }\lim_i\eta^{\bar{\tau}^\fX}_{\tau_{ap}}(\xi_i)=\eta^{\bar{\tau}^\fX}_{\tau_{ap}}(\xi_0)\text{ in }G^{ap}.
\end{align*}
On finite dimensional vector spaces, weak and norm topologies coincide.
Hence we may consider a net in $\fX$ which is weakly converging but with no norm
converging subnets.  For example, let $\fF$ denote the set of all finite subsets of $\fX^*$,
and direct $\fF\times\En$ by $(F,n)\leq(F',n')$ if $F\subseteq F'$ and $n\leq n'$; and
let $\xi_{(F,n)}$ be in $\bigcap_{f\in F}\ker f$ of norm $n$.
If we let $(\xi_i)$ be subnet of a lift of such a net, which converges as above,
then the net $(\eta^{\bar{\tau}^\fX}_{\tau^\fX}(\xi_i))$
cannot converge in the norm topology of $\fX$, hence $(\xi_i)$ cannot converge in $G_{\bar{\tau}^\fX}$.  
If follows that $\bar{\tau}^\fX\subsetneq\bigvee\fS$, and hence
(\ref{eq:nonincl}) holds.

It is now straightforward to conclude that
\[
\wbar{\fT}_{lc}(G)\subsetneq\what{\fT}_{lc}(G)
\subsetneq\wbar{\fT}_u(G)\subsetneq\wbar{\fT}(G).
\]





\begin{thebibliography}{30}

\bibitem[Ars]{arsac}
G. Arsac.
Sur l'espace de Banach engendr\'{e} par les coefficients d'une repr\'{e}sentation unitaire. 
Publ. D\'{e}p. Math. (Lyon) 13 (1976), no. 2, 1--101.


\bibitem[Ban]{banaszczyk}
W. Banaszczyck.
Additive subgroups of topological vector spaces. 
Lecture Notes in Mathematics, 1466. Springer-Verlag, Berlin, 1991.

\bibitem[BYBF]{benyaacovbf}
I. Ben Yaacov, A. Berenstein and S. Ferri.
Reflexive representability and stable metrics. 
Math. Z. 267 (2011), no. 1-2, 129--138. 

\bibitem[BR]{bergelsonr}
V. Bergelson and J. Rosenblatt.
Mixing actions of groups. 
Illinois J. Math. 32 (1988), no. 1, 65--80.

\bibitem[Ber]{berglund}
J.F. Berglund.
Compact semitopological inverse Clifford semigroups. 
Semigroup Forum 5 (1972/73), 191--215. 

\bibitem[BJM]{berglundjm}
J.F. Berglund, H.D. Junghenn and P. Milnes.
{\it Analysis on semigroups. 
Function spaces, compactifications, representations.} John Wiley \& Sons, Inc., New York, 1989.

\bibitem[Bur]{burckel}
R.B. Burckel.  {\it Weakly almost periodic functions on semigroups}. 
Gordon and Breach Science Publishers, New York, 1970.

\bibitem[dLG61]{deleeuwg}
K. De Leeuw and I. Glicksberg.
Applications of almost periodic compactifications. 
Acta Math. 105 (1961), 63--97. 

\bibitem[dLG65]{deleeuwg1}
K. De Leeuw and I. Glicksberg.
The decomposition of certain group representations. 
J. Analyse Math. 15 (1965), 135--192. 

\bibitem[Dye]{dye}
H.A. Dye.
On the ergodic mixing theorem. 
Trans. Amer. Math. Soc. 118 (1965), 123--130. 

\bibitem[Ebe]{eberlein}
W.F. Eberlein.  The point spectrum of weakly almost periodic functions. 
Michigan Math. J. 3 (1955-56), 137--139.

\bibitem[EfRu]{effrosr}
E.G. Effros and Z.-J. Ruan.
Operator spaces. 
London Mathematical Society Monographs. New Series, 23. 
The Clarendon Press, Oxford University Press, New York, 2000.

\bibitem[Eym]{eymard}
P. Eymard.
L'alg\`{e}bre de Fourier d'un groupe localement compact.  
Bull. Soc. Math. France 92 (1964), 181--236.

\bibitem[FeGa]{ferrig}
S. Ferri and J. Galindo.
Embedding a topological group into its WAP-compactification.  
Studia Math. 193 (2009), no. 2, 99--108. 

\bibitem[Gal]{galindo}
J. Galindo.
On unitary representability of topological groups. 
Math. Z. 263 (2009), no. 1, 211--220. 

\bibitem[GHKLMS]{gierzhklms}
G. Gierz, K.H. Hofmann, K. Keimel, J.D. Lawson, M. Mislove, and D.S. Scott.
{\it Continuous lattices and domains}. 
Encyclopedia of Mathematics and its Applications, 93. Cambridge University Press, 2003. 

\bibitem[GlWe]{glasnerw}
E. Glasner and B. Weiss.
On Hilbert dynamical systems. 
Ergodic Theory Dynam. Systems 32 (2012), no. 2, 629--642.

\bibitem[IS07]{ilies}
M. Ilie and N. Spronk.
The spine of a Fourier-Stieltjes algebra.  
Proc. Lond. Math. Soc. (3) 94 (2007), no. 2, 273--301. 

\bibitem[IS12]{ilies1}
M. Ilie and N. Spronk.
Corrigendum: The spine of a Fourier-Stieltjes algebra. 
Proc. Lond. Math. Soc. (3) 104 (2012), no. 4, 859--863. 

\bibitem[Ino]{inoue}
J. Inoue.
On the range of a homomorphism of a group algebra into a measure algebra. 
Proc. Amer. Math. Soc. 43 (1974), 94--98.

\bibitem[Jac]{jacobs}
K. Jacobs.
Ein Ergodensatz f\"{u}r beschr\"{a}nkte Gruppen im Hilbertschen Raum. 
Math. Ann. 128, (1954), 340--349. 

\bibitem[LaLo]{laul}
A.T.-M. Lau and V. Losert.
Complementation of certain subspaces of $L^\infty(G)$ of a locally compact group. 
Pacific J. Math. 141 (1990), no.\ 2, 295--310.

\bibitem[LaLu12]{laulu}
A.T.-M. Lau and J. Ludwig.
Fourier-Stieltjes algebra of a topological group. 
Adv. Math. 229 (2012), no. 3, 2000--2023.

\bibitem[LaLu15]{laulu1}
A.T.-M. Lau and J. Ludwig.
Unitary closure and Fourier algebra of a topological group. 
Studia Math. 231 (2015), no. 1, 1--28.


\bibitem[Mag]{mageira}
A. Mageira.
Some examples of graded C*-algebras. 
Math. Phys. Anal. Geom. 11 (2008), no. 3-4, 381--398. 


\bibitem[May97a]{mayer1}
M. Mayer.
Asymptotics of matrix coefficients and closures of Fourier-Stieltjes algebras. 
J. Funct. Anal. 143 (1997), no. 1, 42--54.

\bibitem[May97b]{mayer2}
M. Mayer.
Strongly mixing groups. Semigroup Forum 54 (1997), no. 3, 303--316.


\bibitem[Meg99]{megrelishvili}
M.G. Megrelishvili.
Operator topologies and reflexive representability. 
Nuclear groups and Lie groups (Madrid, 1999), 197--208, 
Res. Exp. Math., 24, Heldermann, Lemgo, 2001. 

\bibitem[Meg00]{megrelishvili2}
M.G. Megrelishvili.
Reflexively but not unitarily representable topological groups.
Topol. Proc. 25 (Summer), 615--625 (2000).

\bibitem[Meg01]{megrelishvili1}
M.G. Megrelishvili.
Every semitopological semigroup compactification of the group $H_+[0,1]$ is trivial. 
Semigroup Forum 63 (2001), no. 3, 357--370.

\bibitem[Meg08]{megrelishvili3}
M.G. Megrelishvili.
Reflexively representable but not Hilbert representable compact flows and semitopological semigroups. 
Colloq. Math. 110 (2008), no. 2, 383--407. 

\bibitem[Rup90]{ruppert}
W. Ruppert.  
On group topologies and idempotents in weak almost periodic compactifications. 
{\it Semigroup Forum} 40 (1990), no.\ 2, 227--237. 


\bibitem[Rup84]{ruppertB}
W. Ruppert.  {\it Compact semitopological semigroups: an intrinsic theory}. 
Lecture Notes in Mathematics, 1079. Springer, Berlin, 1984.

\bibitem[SvN]{vonneumanns}
I.E. Segal and J. von Neumann.
A theorem on unitary representations of semisimple Lie groups. 
Ann. of Math. (2) 52, (1950), 509--517.

\bibitem[Sht]{shtern}
A.I. Shtern.
Compact semitopological semigroups and reflexive representability of topological groups. 
Russian J. Math. Phys. 2 (1994), no. 1, 131--132. 

\bibitem[SpSt]{spronks}
N. Spronk and R. Stokke.
Matrix coefficients of unitary representations and associated compactifications. 
Indiana Univ. Math. J. 62 (2013), no. 1, 99--148.

\bibitem[Tay]{taylor}
J.L. Taylor.
Measure algebras. 
Conference Board of the Mathematical Sciences Regional Conference Series in Mathematics, No. 16.
American Mathematical Society, Providence, R. I., 1973.

\bibitem[Tel]{teleman}
S. Teleman.
Sur la repr\'{e}sentation lin\'{e}aire des groupes topologiques.  
Ann. Sci. Ecole Norm. Sup. (3) 74, (1957), 319--339.


\bibitem[Usp]{uspenskeij}
V.V. Uspenskeij.
Unitary representability of free abelian topological groups. 
Appl. Gen. Topol. 9 (2008), no. 2,  (2009). 

\bibitem[Var]{varopoulos}
N.Th.\ Varopoulos.
Studies in harmonic analysis. 
Proc. Cambridge Philos. Soc. 60, (1964), 465--516.

\end{thebibliography}
\end{document}